\documentclass[12pt
]{amsart}

\usepackage{amsmath,amssymb,amsthm,amsfonts,url,color,enumitem,dsfont,mathrsfs,hyperref,interval}
\usepackage[T1]{fontenc} \usepackage[utf8]{inputenc}
\usepackage[english]{babel} \usepackage{csquotes}

\hypersetup{plainpages=false,colorlinks,breaklinks}

\newcommand{\R}{\mathbb{R}}
\newcommand{\N}{\mathbb{N}}
\newcommand{\C}{\mathbb{C}}
\newcommand{\M}{\mathcal{M}}
\newcommand{\F}{\mathcal{F}}
\renewcommand{\O}{\mathscr{O}}
\renewcommand{\S}{\mathcal{S}}
\newcommand{\dd}[1]{\ensuremath{\operatorname{d}\!{#1}}}
\DeclareMathOperator{\Op}{\mathrm{Op}^{w}_\mathit{h}}

\DeclareMathOperator{\Tr}{Tr}
\renewcommand{\Im}{\mathrm{Im}\,}
\renewcommand{\Re}{\mathrm{Re}\,}
\newcommand{\Cinf}{\mathscr{C}^\infty}
\newcommand{\deriv}[2]{\frac{\partial#1}{\partial#2}}
\newcommand{\abs}[1]{\left|#1\right|}
\newcommand{\norm}[1]{\left\|#1\right\|}
\newcommand{\pscal}[2]{\langle#1,#2\rangle}
\newcommand{\Lh}{\mathscr{L}_h}

\newcommand{\phy}{\varphi}

	\theoremstyle{plain}
\newtheorem{theorem}{{Theorem}}[section] 
\newtheorem*{theorem*}{Theorem}
\newtheorem{proposition}[theorem]{Proposition}
\newtheorem*{proposition*}{Proposition}
\newtheorem{corollary}[theorem]{Corollary}
\newtheorem*{corollary*}{Corollary}
\newtheorem{lemma}[theorem]{Lemma}
\newtheorem{scholie}[theorem]{Scholium}
\newtheorem*{lemma*}{Lemma}
	\theoremstyle{definition}

\newtheorem*{definition*}{Definition}
	\theoremstyle{remark}
\newtheorem{remark}[theorem]{Remark}
\newtheorem{notation}[theorem]{Notation}
	\theoremstyle{remark}
\newtheorem*{assumption*}{Assumption}

\makeatletter

\@addtoreset{equation}{section}  
\makeatother

\title[Semiclassical exponential localization]{Exponential
  localization\\ in 2D pure magnetic wells} \author{Y. Guedes
  Bonthonneau} \address[Y. Guedes Bonthonneau]{Univ Rennes, CNRS,
  IRMAR --- UMR 6625, F-35000 Rennes, France}
\email{yannick.bonthonneau@univ-rennes1.fr}

\author{N. Raymond} \address[N. Raymond]{Laboratoire Angevin de
  Recherche en Mathématiques, LAREMA, UMR 6093, UNIV Angers, SFR
  Math-STIC, 2 boulevard Lavoisier 49045 Angers Cedex 01, France}
\email{nicolas.raymond@univ-angers.fr}

\author{S. V\~u Ng\d{o}c} \address[S. V\~u Ng\d{o}c]{Univ Rennes,
  CNRS, IRMAR --- UMR 6625, F-35000 Rennes, France}
\email{san.vu-ngoc@univ-rennes1.fr}

\begin{document}

\begin{abstract}
  We establish a magnetic Agmon estimate in the case of a purely
  magnetic single non-degenerate well, by means of the
  Fourier-Bros-Iagolnitzer transform and microlocal exponential
  estimates \emph{à la} Martinez-Sjöstrand.
\end{abstract}

\maketitle

\section{Introduction}

The question of proving the localization of a quantum state has many
mathematical facets. In this article, we investigate the case of the
magnetic Laplacian and prove, under a geometric confinement property on
the magnetic intensity, an Agmon-type localization estimate for
low-lying eigenfunctions of this operator.

\medskip

The interest in the magnetic Laplacian has several origins.  From a
quantum mechanical viewpoint, this operator is a simplified model for
describing the motion of an electron in a strong magnetic field, when
the electrostatic interaction and the relativistic effects are
ignored; its construction is explained for
instance~\cite{doucot-pasquier2004}. In the book ~\cite{FH10}, the
authors recall that the same operator also appears in the
linearization of the Ginzburg-Landau functional in the domain of
superconductors. In Spectral Geometry, the magnetic Laplacian is often
regarded as a natural variant of the Laplace-Beltrami operator when
the symplectic form of the cotangent bundle is twisted by the
pull-back of a closed 2-form from the base manifold, and has proved
important in the study of magnetic geodesics; see for
instance~\cite{kordyukov-taimanov2019}, and references therein.  In
the present study, we consider the magnetic Laplacian on the plane,
which can be defined as follows.

\medskip

When $B$ is a real function on $\R^2$, a semiclassical magnetic
Laplacian associated with $B$ is a family of operators, depending on a
parameter $h>0$, of the form
\begin{equation}\label{eq:laplacian}
  \Lh=(-i h\nabla-\mathbf{A})^2=(h D_1-A_1(x))^2
  +(h D_2-A_2(x))^2\,,\quad D=-i\partial\,.
\end{equation}
Here, $\mathbf{A}=(A_1,A_2)$ is a potential vector associated with
$B$, i.e $B=\partial_1 A_2 -\partial_2 A_1$. Notice that the
semiclassical limit $h\to 0$ is related to the limit of strong
magnetic field $(1/h)B$.

\medskip

The spectral theory of $\Lh$ has received the attention of
several authors; in particular, it follows from \cite{HK11} that if
$B$ is smooth and admits a global non-degenerate minimum, uniquely
attained at some $x\in \R^2$, the bottom of the spectrum of
$\Lh$ (for $h$ small enough) is comprised of multiplicity one eigenvalues
$\lambda_0(h)< \lambda_1(h)\dots$, with
\[
\lambda_j(h) = b_0 h + (C_1 + C_2 j)h^2 +o(h^2)\,.
\]
The corresponding eigenfunctions are concentrated around $x$, in the
sense that their $L^2$ mass outside of a fixed neighbourhood of $x$ is
$\mathscr{O}(h^\infty)$. The purpose of this article is to obtain a
stronger concentration in the case when $B$ is \emph{real
  analytic}.

\subsection{Statement of the result}

From now on, we assume the following:
\begin{enumerate}[label=\textup{(\roman*)}., ref=\textup{(\roman*)}]
  \item\label{as.i} The magnetic field $B$ has a unique minimum $b_0$
    at $x=0$. It is positive, non-degenerate, and not attained at
    infinity ($\liminf B > b_0$).
  \item\label{as.ii} There exists a complex strip
    $\S=\R^2 + i \interval{-a}a^2$ ($a>0$) to which $B$
    can be holomorphically extended as a bounded function.
  \item\label{as.iii} The function
    $(x_1,x_2)\mapsto \int_0^{x_1}\deriv{B(u,x_2)}{x_2}\dd u$ is
    bounded on the strip $\S$.
  \end{enumerate}

For example, $B = 2 - e^{-|x|^2}$ satisfies our assumptions. We will say that a function $f:\R^n\to \R$ \emph{goes linearly to infinity at infinity} if there is a constant $C>0$ such for $|x|>C$, $f> |x|/C$. Our main result is the following exponential localization estimate.
\begin{theorem}\label{thm.main}
  Consider a Lipschitz function $d : \R^2\to \R_+$ with a unique and
  non-degenerate minimum at $0$, $d(0)=0$, and going linearly to
  infinity at infinity, and let $K>0$. Then there exist
  $C,h_0,\epsilon > 0$ such that, for all $h\in(0,h_0)$ and
  $u\in L^2(\R^2)$ such that
  \[
    \Lh u = h\mu u \qquad \text{with} \qquad \mu\leq b_0 +K
    h\,,
  \]
  we have
  \[
    \int_{\R^2}e^{\epsilon d(x)/h}|u(x)|^2\dd x\leq
    C\|u\|^2_{L^2(\R^2)}\,,
  \]
\end{theorem}

Observe that here, the third assumption~\ref{as.iii} seems technical,
and depends on a choice of a system of coordinates, but we have not
been able to remove it. Also note that since we are not trying to
optimize constants in our theorem, the value of $a>0$ in~\ref{as.ii}
is not essential. As a consequence, to lighten notations, we will use
$a>0$ as a generic constant throughout the paper. The size of the
strip on which we are working will be reduced a finite number of
times.

\subsection{Eigenvalues asymptotics}

Strong localization of eigenfunctions, such as the one claimed by
Theorem~\ref{thm.main}, is often a footprint of discrete
spectrum. Indeed, under assumption~\ref{as.i}, it follows from the
usual theory (see~\cite{Avron-Herbst-Simon78}) that below
$h \liminf B$, the spectrum of $\Lh$ is discrete. Let
$\lambda_0(h)\leq \lambda_1(h)\leq \dots \lambda_{\ell}(h)\leq \dots
\leq \liminf B$ be the (possibly finite) sequence of such eigenvalues,
repeated according to their multiplicity. 

The following theorem has been established via a dimensional reduction
in~\cite{HK14} (see also~\cite{HK11} and the review paper~\cite{HK09})
and via a Birkhoff normal form in~\cite{RVN15}. In fact, this theorem
does not require the analyticity of $B$ (i.e assumptions~\ref{as.ii}
and~\ref{as.iii}), but rather $\mathscr{C}^\infty$ bounds on $B$.
\begin{theorem}[\cite{HK11, HK14},~\cite{RVN15}]\label{thm.HKRVN}
  \begin{equation}\label{eq.HK}
    \forall \ell\in\mathbb{N}\,,\quad
    \lambda_{\ell}(h) =
    b_{0}h+\left(2\ell\frac{\sqrt{\det H }}{b_{0}}+
      \frac{{(\mathrm{Tr}\, H^{\frac{1}{2}})}^2}{2b_{0}}\right)h^2+o(h^2)\,,
  \end{equation}
  where $b_0=\min_{\R^2}B$ and $H=\frac{1}{2}\mathrm{Hess}_{(0,0)} B$.
\end{theorem}

\subsection{Complex WKB expansions}
With Theorem~\ref{thm.HKRVN} comes the question of describing the
eigenfunctions. Inspired by the results about the semiclassical
Schrödinger operator with an electric potential, we can wonder whether
the complex version of the famous Wentzel-Kramers-Brillouin (WKB)
Ansatz can be adapted to the magnetic case. Such constructions,
solving formally the eigenvalue problem, are rather rare in the
context of the pure magnetic Laplacian; see however~\cite[VI,
§2]{maslov94}. Their existence has been established for the first time
in a multi-scale framework in~\cite{BHR16} and then in non-degenerate
magnetic wells (\emph{i.e.}, under Assumption~\ref{as.i}) in
\cite{BR19}. Let us recall the latter result (which was generalized to
the Riemannian setting in~\cite{Tho-PhD}).

\begin{theorem}[\cite{BR19}]\label{thm.WKB}
  Under Assumption~\ref{as.i}, and after a
  rotation, we can assume
  \begin{equation}\label{eq.DLB}
    B(x_{1},x_{2})=b_{0}+\alpha x^2_{1}+\gamma x_{2}^2+\O(\|x\|^3)\,,\quad \mbox{ with }0<\alpha\leq\gamma\,.
  \end{equation}
  Let $\ell\in\mathbb{N}$. There exist
  \begin{enumerate}[label=\textup{(\roman*)}.]
  \item a neighborhood $\mathcal{V}$ of $(0,0)$ in $\R^2$,
  \item an analytic function $S$ on $\mathcal{V}$ satisfying
    \[
      \Re S(x)=\frac{b_{0}}{2}
      \left[\frac{\sqrt{\alpha}}{\sqrt{\alpha}+\sqrt{\gamma}}x^2_{1} +
        \frac{\sqrt{\gamma}}{\sqrt{\alpha}+\sqrt{\gamma}}x^2_{2}\right]
      + \O(\|x\|^3)\,,
    \]
  \item a sequence of analytic functions $(a_{j})_{j\in\mathbb{N}}$ on
    $\mathcal{V}$,
  \item a sequence of real numbers $(\mu_{j})_{j\in\mathbb{N}}$
    satisfying
    \[
      \mu_{0}=b_{0}\,,\quad
      \mu_{1}=2\ell\frac{\sqrt{\alpha\gamma}}{b_{0}}+\frac{(\sqrt{\alpha}+\sqrt{\gamma})^2}{2b_{0}}\,,
    \]
  \end{enumerate}
  such that, for all $J\in\mathbb{N}$, and uniformly in $\mathcal{V}$,
  \[
    e^{S/h}\left((-i h\nabla-\mathbf{A})^2-h\sum_{j\geq 0}^{J}\mu_{j}
      h^j\right)\left(e^{-S/h}\sum_{j\geq 0}^{J} a_{j}
      h^j\right)=\O(h^{J+2})\,.
  \]
\end{theorem}
The WKB constructions in~\cite{BHR16, BR19} give a positive answer to
the open problem mentioned by Helffer in~\cite[Section 6.1]{H09}: in
generic situations with pure magnetic field, WKB constructions
corresponding to the low lying spectrum exist. Once the WKB analysis
is done, we want to know to which extent the Ansätze are
approximations of the exact eigenfunctions $u_\ell\in L^2(\R^2)$. It
follows from Theorem~\ref{thm.HKRVN} that, when $h$ is small enough,
the eigenvalues are simple and separated by a gap of order $\sim
h^2$. Thanks to the Spectral Theorem, we deduce that the
WKB Ansätze are approximations in the $L^2$-sense, and even in a weighted $L^2$-space thanks to Theorem \ref{thm.main} (up to taking a smaller $\varepsilon$).

\begin{corollary}\label{cor:approximation-quasimodes}
	Denote by $u_{\ell,J} =\chi(x)e^{-S/h}\sum_{j\geq 0}^{J} a_{j} h^j$,
	with $\chi\in\Cinf_0(\mathcal{V})$, and constant around the
	origin. Then, for fixed $\ell\in \N$ and $\varepsilon>0$ small
	enough, we have for some $\theta\in \R$
	\[
	\|e^{\varepsilon d(x)/h}( e^{i\theta} u_\ell - u_{\ell,J}) \|_{L^2(\R^2)} =
	\O( h^{\frac{J}{2}} ).
	\]
\end{corollary}

(This will be proved at the end of section \ref{sec.5}). In contrast with Theorem \ref{thm.main}, the WKB Ansätze decay like
$e^{-\Re S/h}$ away from the magnetic well; thus, the approximation
should actually hold in a slight perturbation of the weighted space
$L^2(e^{-2\Re S/h})$. Behind this question lies the tunneling effect
problem: such exponential estimates are the heart of the analysis of
the interaction between multiple magnetic wells. The present paper
does not go that far\footnote{The only known (and optimal) result of
  pure magnetic tunnelling has recently been proved in a
  two-dimensional setting in \cite{BHR19} by means of microlocal
  dimensional reductions.}, but establishes that the eigenfunctions
decay like $e^{-\phy(x)/h}$ for some non-negative function $\phy$.
These types of estimates are well-known and proved in the
\emph{electric} Schrödinger operator $-h^2\Delta + V$, where they go by
the name of Agmon (see~\cite{Agmon82, HJ84, Simon84}). As we will see,
the purely magnetic case seems to necessitate a significantly more
advanced strategy, based on the Fourier-Bros-Iagolnitzer (FBI)
transform. (In~\cite{HJ84}, the FBI transform does appear, but not for
proving the exponential localization; it is used in a second step, to
control the asymptotic expansion of eigenvectors and eigenvalues.)

\subsection{Failure of the naive Agmon estimates}
Let us explain formally why the electric strategy fails in giving the
optimal Agmon estimates in the pure magnetic case (see also
\cite[Prop. 4.23]{Ray17} for a slightly different presentation). This
strategy is based on the following formula:
\[
  e^{\varphi/h}(-i h\nabla-\mathbf{A})^2e^{-\varphi/h} =
  (-i h\nabla-\mathbf{A}+\nabla\varphi)^2\,,
\]
where $\varphi$ is bounded and Lipschitz continuous, and on using the
\emph{coercivity} of the real part
\[
  \Re\langle e^{\varphi/h}(-i h\nabla-\mathbf{A})^2e^{-\varphi/h}
  u,u\rangle=\|(-i h\nabla-\mathbf{A})u\|^2-\|\nabla\varphi\, u\|^2\,,
\]
where $u\in\Cinf_0(\R^2)$.  Then, we want to use the magnetic field,
and we notice that
\[
  \|(-i h\nabla-\mathbf{A})u\|^2\geq h\int_{\R^2}B(x)|u|^2\dd x\,,
\]
so that, for all $\lambda\in\R$,
\[
  \Re\left\langle
    \left(e^{\varphi/h}(-i h\nabla-\mathbf{A})^2-\lambda\right)e^{-\varphi/h}
    u,u\right\rangle\geq
  \int_{\R^2}\left(h B(x)-|\nabla\varphi|^2-\lambda\right)|u|^2\dd x\,.
\]
From this last inequality, we see that the only possibility to control
the gradient is that $\varphi$ actually depends on $h$. With the
choice $\varphi=h^{\frac 12}\Phi$, where $\Phi$ is the Agmon distance
(to $0$) associated with the metric
$\left(B-b_0-|\nabla\Phi|^2\right)_+\dd x^2$, we can deduce that, for
eigenvalues such that $\lambda=b_0 h+\O(h^2)$, the
corresponding eigenfunctions $\psi(=e^{-\varphi/h}u)$ satisfy, for $h$
small enough,
\begin{equation}\label{eq.Agmon0}
  \int_{\R^2} e^{2\Phi/h^{\frac 12}}|\psi|^2 \dd x\leq C\|\psi\|^2\,.
\end{equation}
Due to the non-degeneracy of the minimum of $B$, $\Phi$ may be chosen
with a unique and non-degenerate minimum at $0$. Thus,
\eqref{eq.Agmon0} tells us for instance that the ground state is
\emph{a priori} exponentially localized at the scale $h^{\frac 14}$
near the minimum. This is consistent with Theorem~\ref{thm.WKB}, but
much worse than expected. One should be able to prove that the
eigenfunctions, just as the WKB quasi-modes suggest, are localized at
the scale $h^{\frac 12}$ near the minimum. That it is indeed the case
is the main result of this article.

\subsection{Related results}
Some articles have been devoted to the Agmon estimates in the presence
of a magnetic field, but almost always with an additional electric
potential.  For instance, in~\cite{HJ88}, the decay estimates are
inherited from the electric potential and the magnetic field is
considered as a perturbation (see in particular~\cite[p. 629]{HJ88}).
In the same spirit, Agmon estimates are considered in~\cite[Theorem
1.1]{Nakamura99} (see also the closely related articles~\cite{E96,
  Nakamura96}) in the case of an electric well with constant magnetic
field. It is proved that the magnetic field improves the decay of the
eigenfunctions away from the electric well.

We will see in this paper that pure magnetic Agmon estimates at the
\enquote{right} semiclassical scale can be obtained as projections of
microlocal exponential estimates. Our strategy will be inspired by the
ideas of Martinez~\cite{Martinez97} (see also~\cite{Nakamura98},
and~\cite{MS99} in relation with the corresponding WKB analysis). The
fact that we are able to refine this point of view, which is based on
the FBI transform, and to apply it to establish our new magnetic Agmon
estimates, is reminiscent of Sjöstrand's pioneer work on analytic
hypo-ellipticity~\cite{Sjostrand83}.

\begin{remark}
  Throughout our analysis, we will meet some known close links
  between magnetic and Toeplitz operators. These connections are
  described, for instance, in~\cite{CdV95}, or~\cite[Section
  4]{Kordyukov18}. In the context of Toeplitz operators, exponential
  decay estimates of eigenfunctions have been the subject of the
  recent works~\cite[Theorem 1.3]{Kordyukov18} and~\cite[Theorem
  C]{Deleporte18}. In these papers, the semiclassical parameter is of
  the form $h=p^{-1}$, where $p\in\N$ is the degree of tensorization
  of a line bundle.
\end{remark}

\subsection{Organization and strategy}
In Section~\ref{sec.2}, we perform various reductions to put the
magnetic Laplacian in a \enquote{normal form}. Section~\ref{sec.3} is
central in our analysis and is devoted to general properties of the
Fourier-Bros-Iagolnitzer transform. Our presentation closely follows
and, sometimes, completes the one exposed in the book by Martinez
\cite[Chapter 3]{Martinez-book}. This part of the investigation can
also be considered an interpretation of the magnetic Laplacian as a
Toeplitz operator. In Section~\ref{sec.4}, we prove that the FBI
transform of an eigenfunction (with low energy) is exponentially
localized at the scale $h^{\frac 12}$ near
$0\in\R^2\times(\R^2)^*$. We proceed in two steps: firstly, we prove
the exponential microlocalization near the characteristic manifold
(Theorem~\ref{thm.AgmonX1}); secondly, we establish an exponential
localization inside the manifold (Theorem \ref{thm.AgmonX2}). In
Section~\ref{sec.5}, we use the microlocal exponential estimates to
deduce Theorem~\ref{thm.main}.

\section{Normal form}\label{sec.2}

In~\cite{RVN15}, the second and third author constructed a Fourier
integral operator $U_h$ that conjugates the magnetic Laplacian
$\Lh$ to an operator of the form
\[
  \Op(f(\mathcal{H},x_2,\xi_2)) + \O(h^\infty),
\]
microlocally near the characteristic set of $\Lh$, where
$\mathcal{H}=h ^2D_{x_1}^2 + x_1^2$ and $\Op$ denotes the Weyl
quantization. If the symbol $f$ were analytic, and the remainder
$\O(h^\infty)$ improved to $\O(e^{-C/h})$, this
would imply a natural (and probably optimal) exponential estimate on
the bottom eigenfunctions of $\Lh$. However, the FIO $U_h$
is constructed in a relatively non-explicit fashion, including a
generically divergent Birkhoff normal form, and tracking those
estimates down would require quite sophisticated tools of analytic
microlocal analysis.

Since we \enquote{only} want to obtain decay of eigenfunctions and not
the expansion of the bottom eigenvalues of $\Lh$ to any
power of $h$, we will only need a rather crude normal form.
\begin{lemma}\label{lemma:normal-form-L_h}
  Under assumptions \ref{as.i}, \ref{as.ii}, \ref{as.iii}, there
  exists $a>0$ and, for $h>0$, a unitary operator $U_h$ acting on
  $L^2(\R^2,d x)$ such that
\begin{equation}
U_h \Lh U_h^{-1} = \Op( p_{\mathscr{L}}),
\end{equation}
where $p_{\mathscr{L}}$ is an $h$-dependent holomorphic function on
$\R^4 + i[-a,a]^4$, such that
\begin{equation}\label{eq:form-symbol}
p_{\mathscr{L}} = g^{11}\xi_1^2 + 2 g^{12} \xi_1 x_1 + g^{22} x_1^2 + h^2 q,
\end{equation}
where $g^{11}$, $g^{12}$, $g^{22}$ and $q$ are holomorphic, bounded,
and on $\R^4$ they are real valued. Additionally, the $g^{i j}$ are
critical at $0$, and
\begin{equation}
B(x,\xi)= \sqrt{g^{11} g^{22} - (g^{12})^2},
\end{equation}
when restricted to $\R^4$, admits a positive non-degenerate minimum at $0$, uniquely attained, and not attained at infinity.
\end{lemma}

This type of operators, whose symbol is a quadratic form of some
variables, with parameters, was studied by several authors in the
context of hypo-ellipticity in the smooth category
(see~\cite{Boutet-Grigis-Helffer76} and references therein), and in
the analytic category by Sjöstrand in \cite{Sjostrand83}. It would
be interesting to obtain a global version of Sjöstrand's results in
order to give a different proof of Theorem~\ref{thm.main}.

Observe that the exponential decay of eigenfunctions is not preserved
by general unitary operators. However, we will see that $U_h$ can be
explicitly described, so this will not be an issue.

For a given magnetic field $B$, the choice of magnetic potential
$\mathbf{A}$ is not unique. Any other choice $\mathbf{A'}$ differs
from $\mathbf{A}$ by a gradient, \emph{i.e.}
$\mathbf{A'} = \mathbf{A} + \nabla f$. Then, the corresponding
magnetic Laplacian is obtained by conjugating $\Lh$ by the
multiplication operator $u\mapsto e^{if/h}u$, which is unitary, both
pointwise and in $L^2$. Hence, it does not impact
Theorem~\ref{thm.main}. Therefore we may, and we will, assume that
\begin{equation}\label{eq.gauge}
  \mathbf{A}(x)=(0,A_2(x))\,,\qquad A_2(x)=\int_0^{x_1}B(u,x_2) \dd
  u\,,
\end{equation}
Notice that $A_2$ is
real-analytic, admits a holomorphic extension to the strip $\S$, and
its derivatives are bounded on $\S$ according to assumption~\ref{as.iii}.

For $d=2$ or $d=4$ depending on the context and $a>0$, it will be
convenient to set $\mathcal{S}_{a} := \R^d + i [- a, a]^d$.

\subsection{Normal form near the characteristic set}

In this section, we prove Lemma \ref{lemma:normal-form-L_h}. The operator $U_h$ will be decomposed as the composition of a change of variables and a metaplectic operator. Let us start by constructing the change of variable. 

The first idea, which is quite standard, is to choose coordinates in which the
magnetic field is constant \emph{as a 2 form}. In that case, the
natural symplectic structure becomes canonical, and all the magnetic
information is transferred to a variable Riemannian metric. The
guiding model is the case of constant magnetic field and constant
metric, where the magnetic Laplacian takes the form
\[
  \Lh^{\text{ct}} = {(h D_{x_1})}^2 + {(h D_{x_2} - B
    x_1)}^2,
\]
and its bottom eigenvalue is $h B$. The solutions, sometimes called
\emph{zero modes}, to
\[\left(\Lh^{\text{ct}}-h B\right) u = 0\] 
are of the form $e^{-B x_1^2/2h} f$, with $f$ holomorphic, and they
play an important role in the spectral analysis of the magnetic Dirac
operator (see \cite{BLTRS19}).

Coming back to our problem, there are many diffeomorphisms $\kappa$ of
$\R^2$ such that $\kappa_\ast B$ is the canonical $2$ form (Darboux'
Lemma), so we pick the following
\[
  (x_1,x_2) = \kappa(\tilde{x}_1,\tilde{x}_2),\quad \tilde{x}_1 =
  \int_0^{x_1} B(x',x_2) \dd x',\ \tilde{x}_2 = x_2.
\]
That this defines indeed a global diffeomorphism of $\R^2$ is ensured
by assumption~\ref{as.i}.
\begin{lemma}\label{lemma:darboux}
  Under Assumption~\ref{as.i},~\ref{as.ii} and~\ref{as.iii}, $\kappa$
  is a bi-Lipschitz analytic diffeomorphism of $\R^2$ such that
  $\kappa_* \mathbf{B}= \dd{\tilde{x}_1} \wedge \dd{\tilde{x}_2}$ and
  $\kappa_*\mathbf{A} = \tilde x_1 \dd{\tilde x_2}$. Moreover, there
  exists $\lambda\geq 1$ and $a>0$ such that $\kappa$ and
  $\kappa^{-1}$ send $\S_{a'}$ to $\S_{\lambda a'}$ for all
  $a'\in(0,a/\lambda)$.
\end{lemma}

It will be useful to let
\[
\alpha(x_1,x_2)=  \int_0^{x_1} \partial_{x_2} B(u,x_2)d x_1.
\]

\begin{proof}
$\kappa$ is a global diffeomorphism of $\R^2$ because $B$ is positive, and $\kappa^{-1}$ is well defined on $\mathcal{S}_{a}$. Next, there is a $C>0$ such that $|B|\leq C$ and $|\alpha|\leq C$ , so that $\kappa^{-1}$ maps $\mathcal{S}_{a'}$ into $\mathcal{S}_{2Ca'}$ for $0< a'< a$.

We can compute
\[
d_x(\kappa^{-1}) = \begin{pmatrix}
B(x) & \alpha(x) \\ 0 & 1
\end{pmatrix}.
\]
In particular,
\[
(d_x(\kappa^{-1}))^{-1} = \begin{pmatrix}
\frac{1}{B(x)} & - \frac{\alpha(x)}{B(x)}  \\ 0 & 1
\end{pmatrix}.
\]
Since $B\geq b_0>0$ on $\R^2$, and using Assumption~\ref{as.ii}, there
exists $0< a_0 < a$ such that
$|B|^{-1} \leq (\Re B)^{-1} \leq 1/(2b_0)$ on $\mathcal{S}_{a_0}$. In
particular, on $\mathcal{S}_{a_0}$, $(d_x(\kappa^{-1}))^{-1}$ is
bounded.

Around each real point $x$, we can apply the holomorphic local
inversion theorem and deduce that there are $\epsilon_x,\epsilon'_x>0$
such that $\kappa^{-1}$ is a biholomorphism between the ball of radius
$\epsilon_x$ centered at $x$ and its image, which contains the ball of
radius $\epsilon'_x$ around $\kappa^{-1}(x)$. One can give lower
bounds to the constants $\epsilon_x$, $\epsilon_x'$, expressed only in
terms of the $C^2$ norms of $\kappa^{-1}$, and an upper bound on
$(d_x(\kappa^{-1}))^{-1}$. In particular, we can choose them
independent of $x$.

Additionally, if $\kappa^{-1}(x)=\kappa^{-1}(y)$ for some
$x,y\in \mathcal{S}_{a'}$ with $0<a'<a_1$, then $x_2=y_2$, and
$\int_{x_1}^{y_1}B = 0$. Observe that
\[
0=\Re \int_{x_1}^{y_1}B = \int_{\Re x_1}^{\Re y_1} \Re B(t + i \Im x_1, x_2) \dd t - \int_{\Im x_1}^{\Im y_1} \Im B(\Re y_1 + i t, x_2) \dd t.
\]
Since $\Re B \geq b_0/2$ on $\mathcal{S}_{a_1}$, we deduce that $|\Re (x_1 - y_1)| \leq C (a')^2$ for some $C>0$. In particular, according to the argument above, if $a'$ is small enough, this implies that $x_1 = y_1$. For such an $a'>0$, $\kappa^{-1}$ is a biholomorphism between $\mathcal{S}_{a'}$ and $\kappa^{-1}(\mathcal{S}_{a'})$, which satisfies
\[
\mathcal{S}_{a''} \subset \kappa^{-1}(\mathcal{S}_{a'}) \subset \mathcal{S}_{C a'},
\]
for some $a''>0$. Further, $\kappa^{-1}$ is uniformly Lipschitz, and so is its inverse. Taking $\min(a', a'')$ as the new value of $a$ and $\lambda$ the Lipschitz constant of $\kappa$, $\kappa^{-1}$ ends the proof.
\end{proof}

We can associate $\kappa$ with a unitary operator $U_\kappa$ by setting
\[
U_\kappa f(y) = \textup{Jac}(\kappa)^{1/2}\, f(\kappa(y)).
\]
According to Lemma \ref{lem.A} and keeping the same notation, we have
\begin{equation}\label{eq.symbafter0}
U_\kappa\Lh U_\kappa^{-1}=(-i h\nabla_y-\tilde{\mathbf{A}})g^*(-i h\nabla_y-\tilde{\mathbf{A}})-h^2V\,.
\end{equation}
Here, $V$ is explicit in terms of $\kappa$, and $g^\ast$ is the dual Riemannian metric $(\dd\kappa^T \dd\kappa)^{-1}$. Note also that
\begin{equation}\label{eq.symbafter}
(-i h\nabla_y-\tilde{\mathbf{A}})g^*(-i h\nabla_y-\tilde{\mathbf{A}})=\Op(g^*(\eta-\widetilde{\mathbf{A}}(y),\eta-\widetilde{\mathbf{A}}(y)))+\O(h^2)\,,
\end{equation}
where the $\O(h^2)$ comes from the explicit computation of the subprincipal term with the composition formula (the operator in the right hand side is symmetric). 

From explicit expressions for the remainders, we deduce that 
\begin{equation}\label{eq.symbafter1}
U_\kappa\Lh U_\kappa^{-1}=\Op(\|(\xi_1,\xi_2-x_1)\|^2_{g^*}+\O(h^2))\,,
\end{equation}
where the remainder symbol is of the form $h^2 q_1$, $q_1$ holomorphic and bounded on some $\mathcal{S}_{a}$ with $a>0$. Moreover, letting $\tilde{B} = B\circ \kappa$ and $\tilde{\alpha}=\alpha\circ \kappa$, we get
\begin{equation}\label{eq.g*}
\|(\xi_1,\xi_2-x_1)\|^2_{g^*}=\tilde{B}^2\xi_1^2+(\xi_2-x_1+\tilde{\alpha}\xi_1)^2\,.
\end{equation}
We are now almost in the desired form. We consider the following symplectomorphism
\[ 
\kappa_{\M}(x, \xi) = (x + A\xi, \xi)\,, \qquad \text{ where }\quad A =
  \begin{pmatrix}
    0 & 1 \\ 1 & 0
  \end{pmatrix}\,,\quad A^{-1} = A\,.
\]
It is associated with the metaplectic operator $\mathcal{M}$, defined as
\begin{equation}\label{eq:M}
  \M u(x_1,x_2):= \frac{1}{(2\pi h)^2}\int_{\R^4} e^{\frac{i}{h} \Phi(x,y,\xi)} u(y)\dd y \dd \xi\,,
\end{equation}
the phase being given by
\[\Phi(x,y,\xi)=\varphi(x,\xi)-\langle y,\xi\rangle\,,\quad \varphi(x,\xi)=\langle
x- \frac{1}{2} A\xi,\xi\rangle\,,\]
and $\varphi$ being the generating function of $\kappa_{\mathcal M}$. We observe that
\[\M u(x)=(2\pi h)^{-2}\int_{\R^2}e^{\frac{i}{h}\langle x,\xi\rangle}\F_h u(\xi)e^{-\frac{i}{2h}\langle A\xi,\xi\rangle}\dd\xi\,,\]
where we used the semiclassical Fourier transform
\[\F_h u(\xi) = \int_{\R^2} e^{-\frac{i}{h}\pscal{x}{\xi}} u(x) \dd x\,, \quad\F_h^{-1} v(x) = (2\pi h)^{-2}\int_{\R^2} e^{\frac{i}{h}\pscal{x}{\xi}} v(\xi) \dd x\,.\] 
Recalling that
\[ \F_h^{-1}(UV)=\F^{-1}_h(U)\star\F^{-1}_h(V)\,,\]
the operator $\mathcal{M}$ can be written as a convolution operator
\begin{equation}\label{equ:M=K*u}
  \M u = K\star u\,,\qquad K = \F_h^{-1} e^{-\frac{i}{2h}\pscal{A\xi}\xi} = \frac{1}{2\pi
  	h}e^{\frac{i}{2h}\pscal{A x}x}\,,
\end{equation}
where we used the well-known result about the Fourier transform of a quadratic exponential.

 For a symbol $\sigma$ in
$\mathcal{S}'(\R^{4})$ (which is surely the case of the symbols we are
manipulating so far, we have the exact \enquote{Egorov} correspondence
\begin{equation}\label{eq:metaplectic-transform-Weyl}
  \M^{-1}\Op(\sigma)\M = \Op( \sigma\circ\kappa_\mathcal{M})\,.
\end{equation}
It follows that $\M^{-1}U_\kappa \Lh U_\kappa^{-1}\M$ is in the form announced by Lemma \ref{lemma:normal-form-L_h}. It remains to check the conditions on the coefficients. We find that
\[
g^{11} = (\tilde{B}^2 + \tilde\alpha^2)\circ\kappa_\M,\ g^{12} = \tilde{\alpha}\circ \kappa_\M,\ g^{22} = 1.
\]
Then
\begin{equation}\label{equ:B}
B(x,\xi)= \sqrt{ g^{11}g^{22} - (g^{12})^2} =  B(\kappa(\kappa_\M(x,\xi))),
\end{equation}
is suitably non-degenerate according to assumption~\ref{as.i}, and it
remains to check that $\tilde{\alpha}\circ\kappa_\M$ is critical at
$0$. But this is true if $\alpha$ itself is critical at $0$, and this
holds since ($B$ being critical at $0$)
\[
\alpha = x_1^2 \partial_{x_2,x_1}^2 B(0)+\O(x^3).
\]

\qed

It is important to observe that since $\M$ somehow mixes $x$ and $\xi$ variables, it does not preserve exponential decay of functions. However, in a sense to be precised later, we will get decay in \enquote{$x$ and $\xi$}, which is preserved by $\M$. 

In the sequel, it will be convenient to let
\begin{equation}\label{eq:form-symbol-principal}
  p_\M := g^{11}\xi_1^2 + 2 g^{12} x_1\xi_1 + g^{22} x^2_1\,.
\end{equation}
Additionally, we will distinguish variables by setting
$X_1=(x_1,\xi_1)$, $X_2 = (x_2,\xi_2)$.

\subsection{Reduction to a bounded symbol}

Our strategy strongly relies on the presentation of the
Fourier-Bros-Iagolnitzer (in short, FBI) transform given in Martinez'
book~\cite{Martinez-book}. There, many results require that operators
have symbols in the class $S(1)$, which is the space of smooth
functions on phase space that are uniformly bounded, together with all
their derivatives. However, because the magnetic Laplacian is a
differential operator of positive order, its symbol does not belong to
that class. The statements we will use could probably be extended to
the general case of symbols with more general order functions. It is
to avoid this, and concentrate on the essential arguments, that we
have decided to restrict ourselves to the case of a bounded magnetic
field, a situation where we can reduce the problem to a problem in the
$S(1)$ class, as follows.

Initially, the symbol of magnetic Laplacian is polynomial in $\xi$ and
hence belongs to a class with gains of powers of
$\langle \xi \rangle$: locally in $x$,
\begin{equation}\label{eq:estimates-S^2}
  |\partial_x^\alpha \partial_\xi^\beta p_{\kappa}| \leq C_{\alpha,\beta} \langle \xi\rangle^{2 - |\beta|}.
\end{equation}
This still holds after the change of variables $\kappa$. However, the
metaplectic transform $\M$ mixes the $x$ and $\xi$ variables, so that
we do not gain powers of $\xi$ anymore. Recall that for a non-negative
function $m$ on $\R^d$, $S_{\R^d}(m)$ is defined as the set of
functions $\sigma$ that satisfy estimates
\[
  |\partial_x^\alpha \sigma | \leq C_\alpha m,\ \alpha \in \N^d.
\]
If $\sigma$ is holomorphic on a complex strip
$\R^d\subset \S\subset\C^d$, we shall say that $\sigma\in S_\S(m)$ if
\[
  \forall z\in \S,\qquad |\partial_x^\alpha \sigma(z) | \leq C_\alpha
  m(\Re z),\ \alpha \in \N^d.
\]
Also recall what it means for a non-vanishing smooth function $m$ on
$T^\ast \R^2 = \R^4$ to be an \emph{admissible order function}. First,
one requires that $m\in S_{\R^4}(m)$. Second, there is an $N>0$ such that for some
$C>0$ and any $(x,\xi),(x',\xi')\in T^\ast \R^2$,
\begin{equation}\label{eq:temperance}
  \frac{m(x,\xi)}{m(x',\xi')} \leq C \langle (x-x',\xi-\xi')\rangle^N.
\end{equation}
Given two admissible order functions $m$ and $m'$, then $1/m$ and
$mm'$ also are admissible and we have the following result (see for
instance~\cite[Proposition 7.7]{Dimassi-Sjostrand}): if
$\sigma \in S(m)$ and $\sigma'\in S(m')$, then
\begin{equation}\label{eq:product-Weyl}
  \Op(\sigma)\Op(\sigma') =
  \Op\left( \sigma \sigma' + \frac{h}{2i} \{ \sigma, \sigma'\} +
    \O_{S(mm')}(h^2) \right),
\end{equation}
with the usual sign convention
$\{ f,g\} = \partial_\xi f \partial_x g - \partial_x f \partial_\xi
g$. Following the result of Boutet de Monvel-Krée~\cite{BMK}, a
refinement of estimate~\eqref{eq:product-Weyl} shows that if
$\sigma, \sigma'$ had a holomorphic extension to a strip, with uniform
estimates, then the symbol of the product also does, with uniform
estimates. Consider now
\[
  m_{\M}(X_1,X_2)=1+p_{\M}(X_1,X_2)\,.
\]
\begin{lemma}
  Assume that $p_{\mathcal{M}}$ is in the form~\eqref{eq:form-symbol-principal}, with coefficients satisfying
  the conclusion of Lemma~\ref{lemma:normal-form-L_h}. Then $m_{\mathcal{M}}$ is an
  admissible order function, and $p_{\mathcal{M}} \in S_{\mathcal{S}_a}(m_{\mathcal{M}})$ for
  some $a>0$.
\end{lemma}
If these assumptions are satisfied, we will introduce a bounded
spectral parameter $\mu$ and work with
\begin{equation}\label{eq:def-P_h}
\begin{split}
  \mathscr{P}
  & = \Op\left(\frac{1}{1+p_{\mathcal{M}}} \right)\Op( p_{\mathscr{L}}
  - h\mu) \\
  & = \Op\left(\frac{p_{\mathcal{M}}-h\mu}{1 + p_{\mathcal{M}}} +
    \O_{S_{\mathcal{S}_a}(1)}(h^2)\right) = \Op( p_h),
\end{split}
\end{equation}
where $p_{\mathscr{L}} = p_\M + h^2q$, see~\eqref{eq:form-symbol}, 
so that $p_h \in S_{\mathcal{S}_a}(1)$, uniformly with respect to $h$
and $\mu$. 
\begin{remark}
Since the intensity of the magnetic field is
given by
\[
\sqrt{\det \textup{Hess}_{X_1} (p_h)_{|X_1=0}}+\O(h)
\]
(see Equation~\eqref{equ:B}) the fact that it is globally bounded is
actually necessary for obtaining $p_h \in S_{\mathcal{S}_a}(1)$.
\end{remark}
\begin{proof}
  Let us check that $p_{\mathcal{M}} \in S_{\mathcal{S}_a}(m_{\mathcal{M}})$. Since the
  coefficients $g^{i j}$ are in $S_{\mathcal{S}_a}(1)$ for some $a>0$,
  one finds that
  \[
    |\partial^\alpha p_{\M}| \leq C_\alpha(1 + |X_1|^2) \qquad \text{
      on } \mathcal{S}_a.
  \]
  Thus, it suffices to show that there exists $\lambda >0$ such
  that 
  \begin{equation}\label{eq:uniform-lower-bound}
    1 + \Re p_{\M} \geq \lambda (1+ |X_1|^2)\,.
  \end{equation}
  Let us start by
  proving this on $\R^4$.  Note that \eqref{eq:uniform-lower-bound} is satisfied for
  example if $\lambda \leq 1$ and everywhere
  \[
    \lambda \leq \frac{ g^{11} + g^{22} - \sqrt{ (g^{11}-g^{22})^2 + 4
        (g^{12})^2}}{2}\,.
  \]
  Let 
  \[C = \sup \{| g^{11}| +|g^{22}|\}\,,\quad 
  C' = \inf \{g^{11} g^{22} -
  (g^{12})^2\}\,.\]
  The quantity in the right hand side is larger than
  \[2 (g^{11}g^{22} - (g^{12})^2)/(g^{11}+g^{22}) \geq 2 C'/C > 0\,,\]
  uniformly on $T^\ast \R^2$. Now, we turn to the case that
  $(x,\xi) = (\Re x, \Re \xi) + i(u, v)$. Then, we can write
  \begin{align*}
    \begin{split}
      \Re p_{\M} & = \Re\left(g^{11}(\xi_1^2 - v_1^2) + 2g^{12}(\xi_1 x_1 - u_1 v_1) + g^{22}(x_1^2 - u_1^2) \right) \\
      &- 2 \Im \left(g^{11} \xi_1 v_1 + 2 g^{12}(\xi_1 u_1 + x_1 v_1)
        + g^{22}x_1 u_1\right)
    \end{split}\\
                 &\geq \lambda' (\Re X_1)^2 - C a^2(1+ |\Re X_1|),
  \end{align*}
  where $\lambda'$ may be smaller than the $\lambda$ from before, but
  is still non-negative if we assume that
  $\inf\{\Re g^{11}\Re g^{22}- (\Re g^{12})^2\}> 0$ on
  $\R^4 + i\interval{-a}a^4$. Up to taking $a$ small enough, this
  holds.

  Finally, we consider the temperance of the symbol. We already know
  that for some constants $C,C'$,
  \[
    \frac{1+p_\M(x,\xi)}{1+p_\M(x',\xi')} \leq \frac{ 1 + C (X_1')^2}{1
      + C' X_1^2}\,,
  \]
  whence we find
  \[
    \frac{1+p_\M(x,\xi)}{1+p_\M(x',\xi')} \leq \frac{C}{C'}(1+
    \lambda(X_1 - X_1')^2)\,,
  \]
  for $\lambda$ large enough.
\end{proof}

\section{About the FBI transform}\label{sec.3}
Our main tool in this section will be the Fourier-Bros-Iagolnitzer
(FBI) transform. Several versions exist in the literature,
see~\cite{Hitrik-Sjoestrand-mini-cours}; in this paper we
follow~\cite[Chapter 3]{Martinez-book}, and the FBI transform we use
here is defined, for $u\in \mathscr{S}'(\R^2)$, by
\[Tu(x,\xi)=\alpha_h\int_{\R^4}e^{i(x-y)\xi/h}e^{-|x-y|^2/2h}u(y)\dd
  y\,,\quad \alpha_h=2^{-1}(\pi h)^{-\frac 32}\,.\] The $\alpha_h$ is
chosen so that $T$ is isometric from $L^2(\R^2)$ to $L^2(\R^4)$. The
knowledge of $Tu$ implies the knowledge of $u$ via the inversion
formula\footnote{ sometimes called coherent state decomposition; in
  relation with the magnetic Laplacian, it has been used
  in~\cite[Section 2.3]{BHR16}}:
\begin{equation}\label{eq.uTu}
  u(y)=\alpha_h\int_{\R^4}e^{-i(x-y)\xi/h-|x-y|^2/2h}Tu(x,\xi)\dd x\dd \xi = T^\ast T u\,.
\end{equation}

It will be essential later that
\begin{equation}\label{eq:values-in-anti-holomorphic}
  (h(\partial_x -  i \partial_\xi)-  i\xi)  T = 0.
\end{equation}
In other words, $T$ maps $L^2(\R^2)$ into the closed subspace of
$L^2(\R^4)$ of functions of the form $e^{\frac{-\xi^2}{2h}}f(x-i\xi)$,
where $f$ is holomorphic on $\C^2$.

\subsection{Towards a Toeplitz representation}

Since the naive Agmon tactic fails, it seems natural to try and use
weights in phase space that depend on both $x$ and $\xi$. However, it
is not easy to understand the behavior of an operator of the type
$\Op(e^{\psi(x,\xi)/h})$, all the more if $\psi$ was not
bounded. (Although, in the case of a quadratic $\psi$, see the recent
article~\cite{Cob-Sjo-Hit18}). Following the strategy
of~\cite[3.5]{Martinez-book},~\cite{Nakamura95}, we use the FBI
transform to simplify this, as $e^{\psi(x,\xi)/h}$ can be seen as an
multiplication operator on $L^2(\R^4)$. Precisely, let us consider the
following quantity
\begin{equation}\label{eq:quadratic}
  \langle m e^{\psi/h} T\mathscr{P}u,
  e^{\psi/h}Tu\rangle_{L^2(\R^4)}\,,
\end{equation}
where $\mathscr{P}$ is defined in~\eqref{eq:def-P_h}, and $m\in S(1)$
is multiplier (it is not an order function!).  In this section,
$\psi\in S(1)$ and might depend on parameters uniformly with respect
to the $S(1)$-topology, and all the bounds will depend on $\dd\psi$
only.

Since the FBI transform we are using has a quadratic phase, we have an
exact formula
\[
  T\Op(\sigma)=\Op(\sigma_T) T\,,
\]
where $\sigma_T(x,\xi,x^*,\xi^*)=\sigma(x-\xi^*,x^*)$, valid for
$\sigma \in \mathscr{S}'(\R^4)$. From this, we get
\[
  \langle m e^{\psi/h}
  T\mathscr{P}u,e^{\psi/h}Tu\rangle_{L^2(\R^4)}=\langle m e^{\psi/h}
  \mathscr{P}_T Tu,e^{\psi/h}Tu\rangle_{L^2(\R^4)}\,.
\]
We set
\[
  \mathscr{P}^{\psi}_T=e^{\psi/h}\mathscr{P}_{T}e^{-\psi/h}\,,\qquad
  u_\psi=e^{\psi/h}Tu = T_\psi u\,,
\]
so that
\begin{equation*}\label{eq.reformulation}
  \eqref{eq:quadratic}  =\langle m\mathscr{P}^\psi_T u_\psi,u_\psi\rangle_{L^2(\R^4)}\,.
\end{equation*}
Thanks to our analyticity assumption and~\cite[Corollary
5]{Nakamura95} or~\cite[Lemma 3.5.4]{Martinez-book},
$\mathscr{P}^{\psi}_T$ is still a pseudo-differential operator with
symbol in $S(1)$. Its symbol satisfies
\begin{equation}\label{eq:symbol-p-psi}
  p^\psi_h=p_h(x-\xi^*-i\partial_\xi\psi,x^*+i\partial_x\psi)+\O(h^2)\,.
\end{equation}
Since we use the Weyl quantization, we have indeed $\O(h^2)$
and not only $\O(h)$.  Now, we apply~\cite[Theorem
1]{Nakamura95} or~\cite[Theorem 3.5.1]{Martinez-book}, which gives
\begin{equation}\label{eq.q}
  \langle T\mathscr{P}u,me^{2\psi/h}Tu\rangle_{L^2(\R^4)}=\int_{\R^4}p^{\psi}_{h,m}(x,\xi ;h)|u_\psi|^2\dd x\dd\xi+\O(h^2)\|u_\psi\|^2\,,
\end{equation}
with
\begin{equation}
  p^{\psi}_{h,m}(x,\xi ;h):=
  m(x,\xi)p_h(x+2\partial_{\overline z}\psi,\xi-2i\partial_{\overline{z}}\psi)+\O(h)\,.
\end{equation}
Here, we have introduced the complex variable $z = x +i \xi$, and
\[
  \partial_{z} = \frac{1}{2}(\partial_x - i \partial_\xi),\quad
  \partial_{\overline{z}} = \frac{1}{2}( \partial_x + i \partial_\xi).
\]
We stress again that the all the constants in the estimates only
involve $\psi$ via semi-norms of $\dd\psi$ in $S(1)$.

\subsection{Subprincipal term}
In fact, we can even describe the term estimated by $\O(h)$
and we will actually need it. For that purpose, and also for the
convenience of the reader, let us revisit and refine~\cite[Theorem
3.5.1]{Martinez-book}.

\subsubsection{General expression of the subprincipal term}
Let us focus on the proof of~\eqref{eq.q} once we
have~\eqref{eq:symbol-p-psi}. The following proposition shows how to
explicitly write a pseudo-differential operator acting on the range of
$T_\psi$ (in the sense of quadratic forms) as a multiplication
operator modulo $\O(h^2)$.
\begin{proposition}\label{prop.WtoB}
  Consider a symbol $q=q_0(x,\xi, x^*,\xi^*)\in S_{\R^8}(1)$. We have
  \[
    \langle \Op (q)\,
    u_\psi,u_\psi\rangle_{L^2(\R^4)}=\int_{\R^4}(\tilde{q}_0(x,\xi) +
    h\tilde q_1(x,\xi))|u_\psi|^2\dd
    x\dd\xi+\O(h^2)\|u_\psi\|^2_{L^2(\R^4)}\,,
  \]
  where
  \[
    \tilde{q}_0(x,\xi)=q_0(x,\xi,\xi-\partial_\xi\psi,\partial_x\psi)\,,\quad
    \tilde{q}_1(x,\xi)=\frac{1}{2}\left(\{\sigma_f,g\} +
      \{f,\sigma_g\}\right)_{f=g=0}\,,
  \]
  with
  \begin{equation}\label{eq.fg}
    \begin{split}
      f&=x^*-\xi+\partial_\xi\psi\\
      g&=\xi^*-\partial_x\psi\\
      \sigma_f&=\int_0^1
      \partial_{x^*}q_0(x,\xi,\xi-\partial_\xi\psi+t f,\xi^*)\dd t\\
      \sigma_g&=\int_0^1\partial_{\xi^*}q_0(x,\xi,\xi-\partial_\xi\psi,\partial_x\psi+t g)\dd
      t\,.
    \end{split}
  \end{equation}
\end{proposition}
\begin{proof}
  Let us follow the presentation by Martinez.  The computations also
  appear in~\cite{Nakamura95}. We consider
  \[
    r_1(x,\xi,x^*,\xi^*)=q(x,\xi,x^*,\xi^*)-q(x,\xi,\xi-\partial_\xi\psi,\partial_x\psi)\,,
  \]
  By the Taylor formula,
  \[
    r_1=f\sigma_f+g\sigma_g\,.
  \]
  We set $F=\Op f$ and $G=\Op g$. Since we use the Weyl quantization,
  we have
  \[
    \Op( f \sigma_f ) = \frac{1}{2}(F\Op(\sigma_f) + \Op(\sigma_f) F)
    + \O(h^2)\,.
  \]
  (Here, the symbol $f$ is not in $S(1)$, however all its derivatives
  are, which is essential in the computation). Next, we observe that
  Equation~\eqref{eq:values-in-anti-holomorphic} implies that
  $FT_\psi=i G T_\psi$ and deduce
  \[
    \frac{1}{2}\langle (F\Op(\sigma)+\Op(\sigma)F)u_\psi,u_\psi\rangle
    = \frac{i}{2}\langle [ \Op(\sigma),G]u_\psi,u_\psi\rangle\,.
  \]
  Thus (again, since $\dd g \in S(1)$)
  \[
    \langle \Op( f \sigma_f ) u_\psi, u_\psi\rangle = \frac{h}{2}
    \langle \Op(\{\sigma_f, g\}) u_\psi, u_\psi \rangle +
    \O(h^2)\|u_\psi\|^2\,.
  \]
  In the same way, we get
  \[
    \langle \Op( g \sigma_g ) u_\psi, u_\psi\rangle = \frac{h}{2}
    \langle \Op(\{f, \sigma_g\}) u_\psi, u_\psi \rangle +
    \O(h^2)\|u_\psi\|^2\,.
  \]
  Therefore, iterating the argument,
  \[
    \langle\Op r_1
    u_\psi,u_\psi\rangle=\frac{h}{2}\int_{\R^4}\left(\{\sigma_f,g\}+\{f,\sigma_g\}\right)_{f=g=0}|u_\psi|^2\dd
    x\dd\xi+\O(h^2)\|u_\psi\|^2\,.
  \]
\end{proof}

\begin{notation}\label{notation.hat}
  When $a\in S_{\R^4}(1)$, we let
  \[\widehat a(x,\xi)=a(x+2\partial_{\overline
      z}\psi,\xi-2i\partial_{\overline{z}}\psi)\,.\]
\end{notation}

\begin{corollary}\label{cor.multiplication}
  We have
  \begin{equation}
    \langle T\mathscr{P}u,me^{2\psi/h}Tu\rangle_{L^2(\R^4)}=\int_{\R^4}p^{\psi}_{h,m}(x,\xi ;h)|u_\psi|^2\dd x\dd\xi+\O(h^2)\|u_\psi\|^2\,,
  \end{equation}
  where
  \begin{equation}
    p^{\psi}_{h,m}(x,\xi ;h)=m(x,\xi)\widehat{p}_h(x,\xi)+h p^{\psi}_{h,m,1}\,,
  \end{equation}
  and
  \[
    p^{\psi}_{h,m,1}= - 2 \partial_{\overline{z}}m
    \widehat{\partial_{z} p}_h+ m s(x,\xi)\,,
  \]
  with
  \[
    s(x,\xi)=\frac{1}{2}\left(\{\sigma_f,g\}+\{f,\sigma_g\}\right)_{|f=g=0}\,,
  \]
  where we used the notations of Proposition~\ref{prop.WtoB} with
  \[
    q_0(x,\xi,x^*,\xi^*)=
    p_h(x-\xi^*-i\partial_\xi\psi,x^*+i\partial_x\psi)\,.
  \]
\end{corollary}

\begin{proof}
  We apply Proposition~\ref{prop.WtoB} to the pseudo-differential
  operator $\Op q=m\mathscr{P}^\psi_T$. By the composition formula,
  \begin{align*}
    q(x,\xi,x^*,\xi^*)
    & = m q_0(x,\xi,x^*,\xi^*)+hq_1(x,\xi,x^*,\xi^*)+\O(h^2)\,,\\
    \intertext{where}
    q_1(x,\xi,x^*,\xi^*)
    & = (2i)^{-1}\{m(x,\xi),
      p_h(x-\xi^*-i\partial_\xi\psi, x^*+i\partial_x\psi)\}\\
    & = -(2i)^{-1}\partial_{x} m \cdot\partial_\xi p_h +
      (2i)^{-1}\partial_\xi m \cdot \partial_x p_h\,.
  \end{align*}
  We deduce that
  \[
    p^\psi_{h,m} = m(x,\xi)\widehat{p}_h(x,\xi)+
    \frac{i h}{2}\left(\partial_x m \cdot \widehat{\partial_\xi p}_h -
      \partial_\xi m \cdot \widehat{\partial_x p}_h\right)+h s(x,\xi)\,,
  \]
  with
  \begin{align*}
    s 	&= \frac{1}{2}( \{ m\sigma_f, g\} + \{ f , m \sigma_g\} )_{|f=g=0}\\
	&=  \frac{m}{2}(\{\sigma_f,g\} + \{f, \sigma_g\})_{|f=g=0} + \frac{1}{2}( -\widehat{\partial_\xi p}_h \cdot \partial_\xi m - \widehat{\partial_x p}_h \cdot \partial_x m).
  \end{align*}
  But we have
  \begin{align*}
    \frac{i}{2}\left(\partial_x m \cdot \widehat{\partial_\xi p}_h
    - \partial_\xi m \cdot \widehat{\partial_x p}_h\right)
    & + \frac{1}{2}( -\widehat{\partial_\xi p}_h \cdot \partial_\xi m - \widehat{\partial_x p}_h \cdot \partial_x m) \\
    &= \frac{1}{2}\partial_x m \cdot
      \widehat{( -\partial_x p_h  + i \partial_\xi p_h)} +
      \frac{1}{2}\partial_\xi m \cdot
      \widehat{(- i \partial_x p_h - \partial_\xi p_h )}\\
    &= - 2 \partial_{\overline{z}}m \cdot \widehat{\partial_z p}_h.
  \end{align*}
\end{proof}

\subsubsection{Rough estimate of the subprincipal terms}
Let us describe $p_{h,m,1}^\psi$ in the case when $m=m(X_2)\in
S(1)$. Recall that
\[
  p_h = \frac{g^{11}\xi_1^2 - 2 g^{12}\xi_1 x_1 + g^{22}x_1^2 - h \mu
  }{1 + g^{11}\xi_1^2 - 2 g^{12}\xi_1 x_1 + g^{22}x_1^2} + \O(h^2),
\]
where the coefficients $g^{i j}$ are in $S(1)$ on
$\R^4+ i\interval{-a}a^4$, and $\mu\geq 0$. Then we notice that, since $m $ only depends on $z_2$,
\[
  |\partial_{\overline{z}} m \cdot \widehat{\partial_{z} p}_h | =
  \abs{(\partial_{\overline{z_2}} m) \widehat{\partial_{z_2} p}_h}
  \leq C(\min(|X_1|^2,1) + h^2)\,,
\]
and that this term is zero when $m=1$. Also, we observe that a priori,
$s \in S(1)$, so that
\[
  p^{\psi}_{h,m} = m \widehat{p}_h + h m\O(1) + h
  \O(\min(|X_1|^2,1)) + \O(h^3)\,.
\]

\subsubsection{A more accurate description}
When $\psi=\Psi(X_2)$, we can give a more explicit expression for
$s$. It will be convenient to set
\begin{equation}\label{eq.w}
  w(x,\xi,f,g):=p_h(x-2\partial_{\overline z}\psi-g,\xi+2i\partial_{\overline{z}}\psi+f)\,,
\end{equation} 
Then,
\[
  \sigma_f=\int_0^1 \partial_f w(x,\xi, t f,g)\dd t\,,\quad
  \sigma_g=\int_0^1 \partial_g w(x,\xi,0,t g)\dd t\,.
\]
We have
\[
\begin{split}
  \{\sigma_{f},g\}_{f=g=0} &=\sum_{k,j}\{\xi_j,g_k\}\partial_{f_k}\partial_{\xi_j}w(x,\xi,0,0)\\
  			&+\frac{1}{2}\{f_j,g_k\}\partial_{f_k}\partial_{f_j}w(x,\xi,0,0)+\{g_j,g_k\}\partial_{f_k}\partial_{g_j}w(x,\xi,0,0)\,,
\end{split}
\]
and
\[
  \{f,\sigma_{g}\}_{f=g=0}
  =\sum_{k,j}\{f_k,x_j\}\partial_{g_k}\partial_{x_j}w(x,\xi,0,0)+\frac{1}{2}\{f_k,g_j\}\partial_{g_k}\partial_{g_j}w(x,\xi,0,0)\,.
\]
From the expressions of $f$ and $g$, we notice that
$\{\xi_j,g_k\}=-\delta_{j k}$, $\{f_k,x_j\}=\delta_{j k}$ and
\[
  \begin{split}
    \{ g_k, g_j\} 	&= - \partial^2_{\xi_k,x_j}\psi - \partial^2_{\xi_j, x_k}\psi\,,\\
    \{g_k, f_j\} 	&= - \delta_{k,j} + \partial^2_{\xi_k,\xi_j}\psi +
    \partial^2_{x_k,x_j}\psi\,.
  \end{split}
\]
Since $\psi = \Psi(X_2)$, the only non-zero terms involving $\psi$ are
obtained for $j=k=2$. Thus,
\begin{equation}\label{eq.2s1}
  2s(x,\xi)=\left(\sum_k -\partial_{f_k}\partial_{\xi_k}+\frac{1}{2}\partial_{f_k}\partial_{f_k}+\partial_{g_k}\partial_{x_k}+\frac{1}{2}\partial_{g_k}\partial_{g_k}\right)w(x,\xi,0,0)+R_1\,,
\end{equation}
where $R_1 = \O(|\textup{d}^2 \psi| |\textup{d}^2_{X_2} p_h|)$. Let us
look at the first term in the right-hand side of~\eqref{eq.2s1} and
recall~\eqref{eq.w}. Then, we can write it as
\[
  \left(\sum_k
    -\partial_{f_k}\partial_{\xi_k}+\frac{1}{2}\partial_{f_k}\partial_{f_k}+\partial_{g_k}\partial_{x_k}+\frac{1}{2}\partial_{g_k}\partial_{g_k}\right)w(x,\xi,0,0)=-\frac{1}{2}\widehat{\Delta
    p_h}+ R_2\,,
\]
where again, $ R_2= \O(|\textup{d}^2 \psi| |\textup{d}^2_{X_2} p_h|)$,
so that finally,
\[
  s = - \frac{1}{4}\widehat{\Delta p_h} + \O(\min( |X_1|^2 ,
  1)) + \O(h^2).
\]
We can summarize the discussion above in the following.
\begin{scholie}\label{sc.scholie}
  Under the conclusion of Lemma~\ref{lemma:normal-form-L_h}, consider
  $\psi$ bounded with $\dd \psi\in S(1)$ and $m=m(X_2)\in S(1)$. Then
  \[
    \langle m e^{-\psi/h} T \mathscr{P} u, e^{-\psi/h} T u\rangle =
    \int_{\R^4} |u_\psi|^2\left[m \widehat{p}_h + h m s + h r +
      \O(h^2) \right] \dd X_1 \dd X_2\,,
  \]
  where $r,s\in S(1)$ and $\abs{r}\leq C(\min( |X_1|^2,
  1))$.  Moreover, we have the following properties.
  \begin{enumerate}[label=\textup{(\roman*)}.]
  \item When $m=1$, $r=0$
  \item When $\psi= \Psi(X_2)$,
    \[
      s=-\frac{1}{4}\widehat{\Delta p_h}(x,\xi)+ \check{R}\,,
    \]
    where $ \check{R}\in S(1)$ and
    $\check{R}=\O(|\textup{d}^2\Psi|\min (|X_1|^2 + h^2,1))$.
  \end{enumerate}
  Moreover, all estimates are uniform for $h$ small and
  $\dd\psi$ varying in a bounded subset of $S(1)$.
\end{scholie}

Noticing that $\check R$ is zero when $\psi=0$, we get the following.
\begin{proposition}
  When $\sigma\in S(1)$,
  \begin{equation}
  \begin{split}
    \langle \Op(\sigma)u,u\rangle_{L^2(\R^2)}
    &=\langle T\Op(\sigma)u,Tu\rangle_{L^2(\R^4)}\\
    &=\int_{\R^4}\left(\sigma(x,\xi)-\frac{h}{4}\Delta
      \sigma(x,\xi)+\O(h^2)\right)|u_\psi|^2\dd x\dd\xi\,.
    	\end{split}
  \end{equation}
\end{proposition}

\begin{remark}\label{rem.harmonic}
  This classical proposition (see~\cite[Corollary 3.5.7 \& Section
  3.6, Ex. 7]{Martinez-book} and consider also~\cite[Theorem
  13.10]{Z12}) is also true when $\sigma$ is a quadratic form, and in
  this case the remainder $\O(h^2)$ is zero.
\end{remark}

\section{Microlocal Agmon estimates}\label{sec.4}
In this section, we establish Agmon estimates with respect to $X_1$ in
an exponentially weighted space with respect to $X_2$. These estimates
are stated in Theorem~\ref{thm.AgmonX1.0} and~\ref{thm.AgmonX1}. They
imply Theorem~\ref{thm.agmonmicroloc}. In this whole section we will
consider $u\in L^2(\R^2)$ solving the equation
\begin{equation}\label{eq.Pu=0}
  \mathscr{P} u = \Op\left(\frac{1}{1+p_\M}\right) \Op(p_\M - h \mu + h^2 q) u = 0.
\end{equation}
with $p_{\mathscr{L}} = p_\M + h^2 q$ satisfying the conclusion of Lemma
\ref{lemma:normal-form-L_h}, so that the conclusions of
Scholium~\ref{sc.scholie} applies.

\begin{theorem}\label{thm.agmonmicroloc}
  Let $\Psi_1,\Psi_2$ be non-negative Lipschitz functions with a
  unique and non-degenerate minimum at $0$ with minimum value $0$.  We
  also assume that they go linearly to infinity at infinity. We set
  $\psi(x,\xi)=\Psi_1(X_1)+\Psi_2(X_2)$.  Given $K>0$, there exist
  $\varepsilon,h_0,C>0$, such that, for all $h\in(0,h_0)$,
  $\mu\leq b_0+K h$ and $u$ solving~\eqref{eq.Pu=0}, we have
  \[\int_{\R^4}e^{2\varepsilon\psi(x,\xi)/h}|Tu|^2\dd x\dd\xi\leq
    C\int_{\R^4}|Tu|^2 \dd x\dd\xi\, \quad ( = C\norm{u}^2).\]
\end{theorem}

\subsection{Decay away from the characteristic manifold}\label{sec.41}
In this section, we establish the exponential decay of $Tu$ with
respect to $X_1$.

\subsubsection{First estimate}
One will need the following elementary lemma.
\begin{lemma}\label{lem.deformed}
  Recall that
  $p_\M(x,\xi)=g^{11}\xi_1^2 - 2 g^{12} \xi_1 x_1 + g^{22}
  x_1^2$. Then, there exist non-negative numbers
  $\gamma,c_1, c_2, c_3$ such that
  \begin{enumerate}[label=\textup{(\roman*)}.]
  \item for all $|X_1|\geq \gamma$, $\frac{p_\M}{1+p_\M}\geq c_1$,
  \item for all $|X_1|\leq \gamma$,
    $\frac{p_\M}{1+p_\M}\geq c_2|X_1|^2$.
  \end{enumerate}
  If, moreover, $\varepsilon$ is small enough,
  \begin{enumerate}[label=\textup{(\roman*)}.]
  \item for all $|X_1|\geq \gamma$,
    $\Re\frac{\widehat p_\M}{1+\widehat p_\M}\geq c_1$,
  \item for all $|X_1|\leq \gamma$,
    $\Re\frac{\widehat p_\M}{1+\widehat p_\M}\geq \Re \widehat p_\M-c_3|X_1|^4$, and
    $\Re \widehat p_\M\geq c_2|X_1|^2$,
  \end{enumerate}
  where we used Notation~\ref{notation.hat}.
\end{lemma}

\begin{theorem}\label{thm.AgmonX1.0}
  Given $K>0$, there exist $\varepsilon,h_0,C>0$ such that, for all
  $h\in(0,h_0)$, $\mu\leq K$ and $u$ solving~\eqref{eq.Pu=0}, we have
  \[\int_{\R^4}e^{2\varepsilon(\Psi_1(X_1)+\Psi_2(X_2))/h}|Tu|^2\dd
    x\dd\xi\leq C\int_{\R^4}e^{2\varepsilon\Psi_2(X_2)/h}|Tu|^2 \dd
    x\dd\xi\,.\]
\end{theorem}
\begin{proof}
  Assume temporarily that $\Psi_1$ is bounded. Let us use
  Scholium~\ref{sc.scholie} with $m=1$. Then, taking the real part, we
  get
  \[
    \int_{\R^4}(\Re \widehat{p}_h-Ch)|u_\psi|^2\dd x\dd\xi\leq 0\,.
  \]
  Recall
  \[
    p_h(x,\xi)=\frac{ p_\M( x,\xi)-h\mu}{1+ p_\M( x,\xi)}+ \O(h^2)\,.
  \]
  Since $p_\M\geq 0$,
  \[
    \int_{\R^4}\left(\Re\frac{\widehat{p}_\M}{1+\widehat{p}_\M}-C(1+K)h\right)|u_\psi|^2\dd
    x\dd\xi\leq 0\,.
  \]
  Consider $R>0$ and the set
  \[
    J_R=\{X\in\R^4 : |X_1|\geq Rh^{\frac 12}\}\,.
  \]
  We write
  \[
    \begin{split}
      \int_{J_R}&\left(\Re\frac{\widehat{p}_\M}{1+\widehat{p}_\M}-C(1+K)h\right)|u_\psi|^2\dd x\dd\xi \\
      &\leq -\int_{\complement
        J_R}\left(\Re\frac{\widehat{p}_\M}{1+\widehat{p}_\M}-C(1+K)h\right)|u_\psi|^2\dd
      x\dd\xi\,,
    \end{split}
  \]
  and notice
  \[
    \left|\int_{\complement
        J_R}\left(\Re\frac{\widehat{p}_\M}{1+\widehat{p}_\M}-Ch\right)|u_\psi|^2\dd
      x\dd\xi \right|\leq C_R
    h(1+K)\int_{\R^4}e^{2\Psi_2(X_2)/h}|Tu|^2\dd x\dd\xi\,.
  \]
  From Lemma~\ref{lem.deformed}, we get $\tilde c_2>0$ such that on
  $J_R$,
  \[\Re\frac{\widehat{p}_\M}{1+\widehat{p}_\M}-C(1+K)h\geq \tilde c_2
    R^2 h-C(1+K)h\,.\] Choosing $R$ large enough, we get
  \[\int_{J_R}|u_\psi|^2\dd x\dd\xi\leq
    C_R\int_{\R^4}e^{2\Psi_2(X_2)/h}|Tu|^2\dd x\dd\xi\,,\] and then
  \[\int_{\R^4}|u_\psi|^2\dd x\dd\xi\leq
    C\int_{\R^4}e^{2\Psi_2(X_2)/h}|Tu|^2\dd x\dd\xi\,.\] If $\Psi_1$
  is not bounded, we introduce an appropriate cutoff function. For
  example, we apply the previous estimates to
  $\Psi_{1,k}:=\chi(k^{-1}\varepsilon\Psi_1(X_1))\varepsilon\Psi_1(X_1)$
  and send $k$ to $+\infty$. The estimates are independent of $k$
  because $d\Psi_{1,k}$ is uniformly bounded in $S(1)$. Then, we
  conclude with the Fatou lemma.
\end{proof}

\subsubsection{Agmon estimate with multiplier}
Let us now add a \emph{multiplier} in the previous estimate. This can
be done modulo $\O(h)$.
\begin{theorem}\label{thm.AgmonX1}
  Consider $m=m(X_2)$ non-negative with $m\in S(1)$. Then, for $M>0$,
  there exist $\epsilon,h_0,C>0$ such that, for all $h\in(0,h_0)$,
  $\mu\leq M$ and $u$ solving~\eqref{eq.Pu=0}, we have
  \[
    \int_{\R^4}me^{2\varepsilon(\Psi_1(X_1)+\Psi_2(X_2))/h}|Tu|^2\dd
    x\dd\xi\leq C\int_{\R^4}(m+h)e^{2\varepsilon\Psi_2(X_2)/h}|Tu|^2
    \dd x\dd\xi\,.
  \]
\end{theorem}

\begin{proof}
  We use again Scholium~\ref{sc.scholie}, this time without assuming
  that $m=1$. We have
  \begin{equation*}
    \Re\int_{\R^4}p^{\psi}_{h,m}(x,\xi ;h)|u_\psi|^2\dd x\dd\xi=\O(h^2)\|u_\psi\|^2\,,
  \end{equation*}
  so that, with Theorem~\ref{thm.AgmonX1.0},
  \begin{equation*}
    \int_{\R^4}\Re p^{\psi}_{h,m}(x,\xi ;h)|u_\psi|^2\dd x\dd\xi\leq Ch^2\int_{\R^4}e^{2\varepsilon\Psi_2(X_2)/h}|Tu|^2\dd x\dd\xi\,.
  \end{equation*}
  Then, by Scholium~\ref{sc.scholie},
  \begin{multline*}
    \int_{\R^4}\left(m\Re \widehat{ p}_{h}(x,\xi ;h)-Chm\right)|u_\psi|^2\dd x\dd\xi\leq Ch\int_{\R^4}|X_1|^2|u_\psi|^2\dd x\dd\xi\\
    +Ch^2\int_{\R^4}e^{2\varepsilon\Psi_2(X_2)/h}|Tu|^2\dd x\dd\xi\,.
  \end{multline*}
  Using again Theorem~\ref{thm.AgmonX1.0} with a smaller $\varepsilon$
  to absorb the $|X_1|^2$ term,
  \begin{equation*}
    \int_{\R^4}\left(m\Re \widehat p_{h}(x,\xi ;h)-Chm\right)|u_\psi|^2\dd x\dd\xi\leq Ch^2\int_{\R^4}e^{2\varepsilon\Psi_2(X_2)/h}|Tu|^2\dd x\dd\xi\,.
  \end{equation*}
  Then, the analysis the same lines as in the proof of
  Theorem~\ref{thm.AgmonX1}. The same splitting of the integral in the
  left-hand-side gives the conclusion.
\end{proof}

\subsection{Subprincipal decay estimates}\label{sec.42}
Let us now prove an exponential estimate with respect to all the phase
space variables. In the previous section, we essentially used the
ellipticity of the operator outside of the characteristic set. The
results, while new in this precision as far as we know, are not
surprising. However, in this section, we have to understand what is
happening directly on the characteristic set, i.e understand in detail
the \emph{subprincipal} terms. This is a much finer analysis.  At the
microlocal level, the computations are similar to the ones
in~\cite{Sjostrand83}; however, instead of using the Boutet de Monvel
calculus for polynomial operators, we directly use the invertibility
of an effective harmonic oscillator.

\begin{theorem}\label{thm.AgmonX2}
  For $M>0$, there exist $\varepsilon,h_0,C>0$ such that, for all
  $h\in(0,h_0)$, $\mu\leq b_0+M h$ and $u$ solving~\eqref{eq.Pu=0}, we
  have
  \[\int_{\R^4}e^{2\varepsilon\Psi_2(X_2)/h}|Tu|^2\dd x\dd\xi\leq
    C\int_{\R^4}|Tu|^2 \dd x\dd\xi\,.\]
\end{theorem}

\begin{proof}
  This time, $\psi= \Psi(X_2)$. Let us use Scholium~\ref{sc.scholie}
  again with $m=1$. We get
  \[
    \int_{\R^4}\Re\left(\widehat{ p}_h-\frac{h}{4}\widehat{\Delta
        p_h}-\check R-Ch^2\right)|u_\psi|^2\dd x\dd\xi\leq 0\,.
  \]
  With Theorem~\ref{thm.AgmonX1.0}, we can estimate $\check R$ and get
  \[
    \int_{\R^4}\Re\left(\widehat{ p}_h-\frac{h}{4}\widehat{\Delta p_h}
      - Ch^2\right)|u_\psi|^2\dd x\dd\xi\leq 0\,.
  \]
  Observe that
  \begin{equation}\label{eq:decomp-p_h}
    \widehat{p}_h + \O(h^2) = \frac{\widehat{p}_\M - h \mu}{1+\widehat{p}_\M}=\widehat{p}_\M - h \mu + h\mu \frac{\widehat{p}_\M}{1+\widehat{p}_\M}-\frac{\widehat{p}_\M^2}{1+\widehat{p}_\M}\,.
  \end{equation}
  The fourth term in the right-hand side is $\O(\min(|X_1|^4,1))$, and
  can be absorbed using Theorem~\ref{thm.AgmonX1.0}, and replaced by a
  $\O(h^2)$. The third term can also be absorbed in the same
  fashion, and replaced by $\O(h^2 \mu)$. We deduce that
  \[
    \int_{\R^4}\left(\Re\widehat{p}_\M -b_0
      h-\frac{h}{4}\Re\widehat{\Delta
        p_h}-C(1+K)h^2\right)|u_\psi|^2\dd x\dd\xi\leq 0\,.
  \]
  Using Equation~\eqref{eq:decomp-p_h} to estimate the contribution
  from $\widehat{\Delta p_h}$, and using the same arguments,
  \[
    \int_{\R^4}\left(\Re\widehat{p}_\M-b_0
      h-\frac{h}{4}\Re\widehat{\Delta_{X_1}
        p_\M}-C(1+K)h^2\right)|u_\psi|^2\dd x\dd\xi\leq 0\,,
  \]
  Now, we will approximate $\widehat{p}_M$ by a quadratic form in
  $X_1$, with coefficients depending only on $X_2$. To this end, let
  \[
    \mathcal{Q}_{X_2}(X_1) := \left[\widehat{g^{11}}_{|X_1=0}\right]
    \xi_1^2 - 2\left[\widehat{g^{12}}_{|X_1=0}\right]\xi_1 x_1 +
    \left[\widehat{g^{22}}_{|X_1=0}\right]x_1^2.
  \]
  (Observe that since $\psi$ does not depend on $X_1$, $\ \widehat{}\ $ and
  differentiation in $X_1$ commute). Since the coefficients $g^{i j}$
  are assumed to be critical at $0$, and $d\psi(0)=0$, we find
  \[
    \begin{split}
      \Re\widehat{p}_\M - \frac{h}{4}\Re\widehat{\Delta_{X_1} p_\M} &= \Re \mathcal{Q}_{X_2}(X_1) - \frac{h}{2} \Tr\Re \mathcal{Q}_{X_2} \\
       + \O(&|X_1|^2(\min( |X_1|^2 +|X_2|^2, 1) + h |X_1|^2 +
      h |X_1|\min( |X_2|,1)))
    \end{split}
  \]
  Using Theorem~\ref{thm.AgmonX1}, we can absorb
  $\O( h^k |X_1|^{2\ell}\min( |X_2|^2,1))$ and replace it by
  \[
    \O( h^{k+\ell}(\min(|X_2|^2,1) + h)).
  \]
  Therefore, using also
  $|X_1| |X_2| \leq \varepsilon^{-1}|X_1|^2 + \varepsilon|X_2|^2$, we
  get
  \[
    \begin{split}
      \int \left(\Re \mathcal{Q}_{X_2}(X_1) - b_0 h - \frac{h}{2} \Tr\Re \mathcal{Q}_{X_2}  - C\varepsilon h\min(|X_2|^2,1)\right)& |u_\psi|^2 \dd X_1 \dd X_2 \leq \\
      (1 + \varepsilon^{-1}+K)h^2\int|&u_\psi|^2 \dd X_1 \dd X_2.
    \end{split}
  \]

  For fixed $X_2$, we recognize the Bargmann symbol of the
  \enquote{harmonic oscillator} in $X_1$ (see
  Remark~\ref{rem.harmonic}) and thus
  \[
    \int_{\R^2} \left(\Re \mathcal{Q}_{X_2}(X_1) - \frac{h}{2} \Tr \Re
      \mathcal{Q}_{X_2}\right)|u_\psi|^2 \dd X_1 \geq h
    \sqrt{\det{\Re\mathcal{Q}_{X_2}}}\int |u_\psi|^2 \dd X_1.
  \]
  So that
  \[
    \int_{\R^4} \Big(\sqrt{\det{\Re\mathcal{Q}_{X_2}}} - b_0 -
    C\varepsilon\min(|X_2|^2,1) - C(1+ \varepsilon^{-1}+K)h
    \Big)|u_\psi|^2\dd x\dd\xi\leq 0\,.
  \]
  Recall now that $B= \sqrt{\det \mathcal{Q}}$, so that
  \[
    \sqrt{\det{\Re\mathcal{Q}}} = B(1 + \O( \Tr \Re
    \mathcal{Q}^{-1}\Im\mathcal{Q}))= B (1 + \varepsilon
    \O(\min(|X_2|^2,1))) .
  \]
  Under the conclusion on
  Lemma~\ref{lemma:normal-form-L_h}, we get the
  estimate
  \[
    \int_{\R^4} \Big( \min(|X_2|^2, 1) (1 - C\epsilon) - C(1+
    \varepsilon^{-1}+K)h \Big)|u_\psi|^2\dd x\dd\xi\leq 0\,.
  \]
  The conclusion follows from the usual Agmon arguments, and again the
  fact that the constant only depend on derivatives of $\psi$.
\end{proof}

\section{Space exponential decay}\label{sec.5}
We are now in position to prove Theorem~\ref{thm.main}. Let
$\hat u\in L^2(\R^2)$ such that $\Lh \hat u = h\mu \hat u$
and let $u = \M^{-1} U_\kappa \hat u= \M^* U_\kappa\hat u$. We have
$ \mathscr{P} u = 0$, see Equation~\eqref{eq.Pu=0}.

\begin{remark}
Following \cite[Theorem 4.1.2]{Martinez-book}, with Theorem
\ref{thm.agmonmicroloc}, one could deduce (up to technicalities) that, if $K$ is a
compact set away from $0$, we have
\[
\| \hat{u}\|_{L^2(K)}=\O(e^{-c/h})\,
\]
for some $c>0$.
Below, one will get a more explicit result. Already observe that we can drop the factor $U_\kappa$. Indeed, since $\kappa$ is uniformly bi-Lipschitz, it preserves spatial exponential decay. So we can concentrate on $\tilde{u}:= \M u = U_\kappa \hat{u}$.
\end{remark}

With the notation of Theorem~\ref{thm.agmonmicroloc}, we have,
for $\epsilon$ small enough,
\[
  Tu \in L^2_{\epsilon\psi}(\R^4), \quad \text{ where }
  L^2_{\epsilon\psi}(\R^4) := L^2(\R^4; e^{\epsilon\psi}\dd x \dd
  \xi)\,,
\]
with a uniform bound:
$\norm{Tu}_{ L^2_{\epsilon\psi}(\R^4)} \leq C \norm{u}_{L^2(\R^2)}$,
where $C$ does not depend on $h$.  From this exponential decay in
phase space, we wish to obtain exponential decay in the position
variable $x$ for $\tilde u$. We start with the inversion
formula~\eqref{eq.uTu}:
\[
  \hat u = \M T^* (Tu).
\]
Let $\phy$ be a non-negative Lipschitz function, going linearly to
infinity at infinity, having a unique and non-degenerate minimum at
the origin, with minimal value $0$ (let us call these functions
\emph{admissible weights}). We would like to obtain a uniform bound
$\norm{e^{\epsilon' \phy}\hat u}_{L^2(\R^2)}\leq C
\norm{\hat{u}}_{L^2(\R^2)}$, for some $\epsilon'>0$ small
enough. Thus, it is enough to prove that the operator
\begin{equation}\label{equ:bounded-operator}
  \M T^* : L^2_{\epsilon\psi}(\R^4) \to L^2_{\epsilon'\phy}(\R^2)
\end{equation}
is uniformly bounded with respect to $h\in\interval[open left]0{h_0}$.
\begin{lemma}\label{lemm:phy-equiv}
  Let $\phy_1$, $\phy_2$ be admissible weights on $\R^d$. Then there
  exists $C>0$ such that
  \begin{equation}
    C^{-1}\phy_1 \leq \phy_2 \leq C\phy_1\,.\label{equ:phy-equiv}
  \end{equation}
\end{lemma}
\begin{proof}
  By the Taylor formula at the origin, the
  estimate~\eqref{equ:phy-equiv} is valid in a neighborhood of $0$. By
  the linear behavior at infinity, it is valid outside of a
  compact set. On the remaining compact subset of
  $\R^d\setminus\{0\}$, it is enough to use that the range of $\phy_j$
  is a compact interval of $\interval[open left]0{+\infty}$.
\end{proof}
A consequence of the lemma is that the choice of $\phy$ and $\psi$
in~\eqref{equ:bounded-operator} is not relevant, as long as we don't
seek the optimal constants and are allowed to play with
$\epsilon,\epsilon'$.

\begin{proposition}\label{prop.spacedecay}
  Given an admissible weight $\phy$ on $\R^2$, there exists an
  admissible weight $\psi$ on $\R^4$ of the form required by
  Theorem~\ref{thm.agmonmicroloc}, and a constant $C>0$ independent of
  $h$, such that for all $\epsilon' \leq \epsilon/C$,
  $\epsilon\leq 1$, the operator $\M T^*$ defined
  in~\eqref{equ:bounded-operator} is bounded by $\O(1)$.

  As a consequence, there exists $\epsilon'_0$ such that if
  $\epsilon' \leq \epsilon'_0$, then there exists $C>0$ such that
  \[
    \int_{\R^2} e^{\epsilon'\phy(x)/h} \abs{\hat u(x)}^2 \dd x \leq C 
    \norm{\hat u}^2_{L^2(\R^2)}.
  \]
\end{proposition}

\begin{proof}
  By Lemma~\ref{lemm:phy-equiv}, we can always change the function
  $\phy$, so we will pick a convenient one. First, consider the
  $\mathscr{C}^1$ Lipschitz function $f$ defined by $f(\rho)=\rho^2$
  if $\rho\in[0,1]$ and $f(\rho)=2\rho-1$ if $\rho\geq 1$. Notice
  that
  \begin{equation}\label{eq.f}
  \forall \rho\geq 0\,,\quad f(2\rho)\leq 4 f(\rho)\,.
  \end{equation}
  We define now
  the admissible weight $\varphi(x):=f(|x|)$, $x\in\R^2$.

  A formula for $\M T^*$ can be obtained from the action of the FBI
  transform $T$ on arbitrary metaplectic operators,
  see~\cite[3.4]{Martinez-book}; here we derive it explicitly. We have
  \[
    Tu(x,\xi) = \alpha_h e^{\frac{-\xi^2}{2h}} (L \star u)(z), \quad\text{
      with } L(x) = e^{\frac{-x^2}{2h}}, \quad z := x - i\xi \in
    \C^2\,.
  \]
  From~\eqref{equ:M=K*u} we obtain
  $T\M^* u (x,\xi) = \alpha_h e^{\frac{-\xi^2}{2h}} ((L\star
  \overline{K})\star u) (z)$, where $L\star \overline{K}$ is a complex Gaussian
  that can be computed explicitly, using in particular that
  $(I+i A)^{-1} = \tfrac{1}{2}(I-i A)$:
  \[
    L\star \overline{K}  (y) = \sqrt 2 \pi h
    e^{\frac{-1}{4h}\pscal{(I+i A)y}y}\,,\quad A =
    \begin{pmatrix}
    0 & 1 \\ 1 & 0
    \end{pmatrix}\,.
  \]
  Hence
  \[
    (T\M^* u) (x,\xi) = \tilde\alpha_h e^{\frac{-\xi^2}{2h}} \int_{\R^2}
    e^{\frac{-1}{4h}\pscal{(I+i A)(z-y)}{z-y}} u(y) \dd y\,,\qquad\tilde\alpha_h= \frac{\alpha_h}{\sqrt{2}}\,,
  \]
  and therefore, taking the
  adjoint, we have for $v\in L^2_{\epsilon\psi}(\R^4) $
  \[
    (\M T^*) v (y) = \tilde\alpha_h \int_{\R^4} e^{\frac{-\xi^2}{2h}}
    e^{\frac{-1}{4h}\pscal{(I-i A)(\bar z-y)}{\bar z-y}} v(x,\xi) \dd x
    \dd \xi\,,\quad \bar z = x + i \xi\,.
  \]
  Let $K_{\M T^*}(y,x,\xi)$ be the Schwartz kernel of
  $e^{\frac{\epsilon'\phy}{h}} \M T^* e^{-\frac{\epsilon\psi}{h}}$,
  viewed as an operator $L^2(\R^4)\to L^2(\R^2)$, \emph{i.e.}
  \[
    K_{\M T^*}(y,x,\xi) = \tilde\alpha_h e^{\frac{\epsilon'\phy(y)}{h}
      - \frac{\abs{\xi}^2}{2h} - \frac{1}{4h}\pscal{(I-i A)(\bar
        z-y)}{\bar{z}-y} - \frac{\epsilon\psi(x,\xi)}{h}}\,.
  \]
  We have

\begin{align*}
  \Re \pscal{(I-iA)(\bar z-y)}{\bar z-y}
  & = \abs{x-y}^2 - \abs{\xi}^2 - 2 \pscal{A\xi}{x-y}\\
  & = \abs{A\xi - (x-y)}^2 - 2 \abs{\xi}^2,
\end{align*}
where in the second line we used $\abs{A\xi}^2 =
\abs{\xi}^2$. Therefore,
\[
  \abs{K_{\M T^*}(y,x,\xi) } \leq {\tilde\alpha}_h
  e^{\frac{\epsilon'\phy(y)}{h} - \frac{\abs{A\xi - (x-y)}^2}{2h} -
    \frac{\epsilon\psi(x,\xi)}{h}}.
\]
Let us choose now, as we may,
$\psi(x,\xi) := \phy(x) + \abs{\xi}^2/\langle \xi \rangle$. Indeed,
$\psi$ is not of the form $\Psi_1+\Psi_2$, but is bounded from below
by a function of this form (this can be written explicitly, or by
invoking Lemma~\ref{equ:phy-equiv}). By convexity of $\phy$,
\[\phy(y) \leq \tfrac{1}{2}\phy(2x) + \tfrac{1}{2}\phy(2(y-x))\,,\] 
and hence
\[
  \abs{K_{\M T^*}(y,x,\xi) } \leq {\tilde\alpha}_h
  e^{\frac{\epsilon'\phy(2x)}{2h} - \frac{\epsilon\phy(x)}{h} +
    \frac{\epsilon'\phy(2(y-x))}{2h} - \frac{\abs{x-y - A\xi}^2}{2h} -
    \frac{\epsilon\abs{\xi}^2}{h\langle \xi \rangle}}.
\]
If $\epsilon' \leq \epsilon/2$ then, by \eqref{eq.f},
$\epsilon'\phy(2x)/2 \leq \epsilon\phy(x)$ for all $x\in\R^2$ so that
\begin{equation}\label{equ:K}
  \abs{K_{\M T^*}(y,x,\xi) } \leq {\tilde\alpha}_h
  e^{ \frac{\epsilon'\phy(2(y-x))}{2h} - \frac{\abs{x-y - A\xi}^2}{2h} -
    \frac{\epsilon\abs{\xi}^2}{h\langle \xi \rangle}}.
\end{equation}
We wish to conclude on the $L^2$ continuity of $\M T^*$ by applying
the Schur lemma. For given $(x,\xi)$, we make a change of variables to get
\[
  \int_{\R^2} \abs{K_{\M T^*}(y,x,\xi) } \dd y = \int_{\R^2} \abs{K_{\M
      T^*}(y+x+A\xi,x,\xi) } \dd y.
\]
From~\eqref{equ:K} we have
\[
  \abs{K_{\M T^*}(y+x+A\xi,x,\xi) } \leq {\tilde\alpha}_h
  e^{\frac{\epsilon'\phy(2(y+A\xi))}{2h} - \frac{\abs{y}^2}{2h}
    -\frac{\epsilon\xi^2}{h\langle \xi \rangle}}\,.
\]
Using again the convexity of $\phy$,
\begin{equation}
  \frac{\epsilon'\phy(2(y+A\xi))}{2h} - \frac{\abs{y}^2}{2h}
  -\frac{\epsilon\xi^2}{h\langle \xi \rangle} \leq
  \frac{\epsilon'\phy(4y)}{4h}+ \frac{\epsilon'\phy(4A\xi)}{4h}  - \frac{\abs{y}^2}{2h} 
 -\frac{\epsilon\xi^2}{h\langle \xi
    \rangle}.\label{equ:convexity-A}
\end{equation}
From Lemma~\ref{equ:phy-equiv}, there exists $C_\phy>0$ such that,
$\phy(4A\xi) \leq C_\phy\abs{\xi}^2/{\langle \xi \rangle}$. Hence, if
$\epsilon' \leq 4C_\phy\epsilon$, we get
$ \frac{\epsilon'\phy(4A\xi)}{4h} -\frac{\epsilon\xi^2}{h\langle \xi
  \rangle} \leq 0$, and
\[
  \int_{\R^2} \abs{K_{\M T^*}(y+x+A\xi,x,\xi) } \dd y \leq {\tilde\alpha}_h
  \int_{\R^2} e^{\frac{\epsilon'\phy(4y)}{2h} - \frac{\abs{y}^2}{2h} }\dd
  y\,.
\]
Using Laplace's method, the integral on
the right-hand side is $\O(h)$ provided $\epsilon'<1/16$. Hence,
\[\int_{\R^2} \abs{K_{\mathcal{M}T^*}(y,x,\xi) } \dd y \leq C h {\tilde\alpha}_h\,.\]
On the other hand, for a given $y$, we make an analogous change of
variables:
\[
  \int_{\R^4} \abs{K_{\M T^*}(y,x,\xi) } \dd x \dd \xi = \int_{\R^4} \abs{K_{\M
      T^*}(y,x+y+A\xi,\xi) } \dd x \dd \xi
\]
and write~\eqref{equ:K} as
\[
  \abs{K_{\M T^*}(y,x+y+A\xi,\xi) } \leq {\tilde\alpha}_h
  e^{\frac{\epsilon'\phy(2(x+A\xi))}{2h} - \frac{\abs{x}^2}{2h}
    -\frac{\epsilon\xi^2}{h\langle \xi \rangle}}
\]
Applying Equation~\eqref{equ:convexity-A} with $y$ replaced by $x$,
and choosing $\epsilon' \leq 2C_\phy\epsilon$, gives
\[
  \int_{\R^4} \abs{K_{\M T^*}(y,x+y+A\xi,\xi) } \dd x \dd\xi \leq
  {\tilde\alpha}_h \int_{\R^2} e^{-\frac{\epsilon\xi^2}{2h\langle \xi
      \rangle}} \dd \xi \int_{\R^2} e^{\frac{\epsilon'\phy(4x)}{4h} -
    \frac{\abs{x}^2}{2h} }\dd x\,.
\]
Using that both integrals are $\O(h)$, we have
\[
\int_{\R^4} \abs{K_{\M T^*}(y,x+y+A\xi,\xi) } \dd x \dd\xi \leq Ch^2
{\tilde\alpha}_h \,.
\]
Hence, the Schur lemma gives
$\M T^* = \O({\tilde\alpha}_h h^{3/2}) = \O(1) :
L^2_{\epsilon\psi}(\R^4) \to L^2_{\epsilon'\phy}(\R^2)$.
\end{proof}
With this Proposition~\ref{prop.spacedecay}, the proof of
Theorem~\ref{thm.main} is complete. It would be very interesting to
investigate the optimality of $(\phy,\epsilon')$ for which
Proposition~\ref{prop.spacedecay} holds, in particular by relating
$\phy''(0)$ to the behaviour of the magnetic field at the origin.

\begin{proof}[Corollary \ref{cor:approximation-quasimodes}]
In $L^2(\R^2)$, we can decompose $u_{\ell,J}= \alpha e^{i\theta} u_\ell + w$ with $w$ orthogonal to $u_\ell$, $\theta\in\R$ and $\alpha\geq 0$. Since the eigenvalues of $\mathscr{L}_h$ are $h^2$ separated, the Spectral Theorem implies that $\|(\mathscr{L}_h - \lambda_\ell(h))u_{\ell,J}\|\geq C h^2 \|w\|$, so $\|w\|\leq C h^J$. In particular, $1-\alpha \leq (C h^J)^2/2$. We then get
\[
\| u_{\ell,J}- e^{i\theta} u_\ell\|_{L^2}^2 = 2(1-\alpha) \leq (Ch^J)^2.
\]
Now, we turn to exponentially weighted spaces. We consider $\epsilon>0$ small enough so that Theorem \ref{thm.main} applies to $2\epsilon$. Then we observe that
\[
\|u_{\ell,J}-e^{i\theta} u_\ell\|_{L^2(e^{\epsilon d(x)/h}dx)} \leq \| u_{\ell,J}-e^{i\theta} u_\ell\|_{L^2}^{1/2} \|u_{\ell,J}-e^{i\theta} u_\ell\|_{L^2(e^{2\epsilon d(x)/h}dx)}^{1/2}\, .
\]
\end{proof}

\appendix
\section{Change of variables}\label{sec.A}
Since $\Lh$ is invariantly defined by the 2-form $\mathbf{B}$ and the
Riemannian metric on $M$, its principal and subprincipal Weyl symbols
are well defined, which implies that a change of variables like the
one defined in Lemma~\ref{lemma:darboux} and used in
Lemma~\ref{lemma:normal-form-L_h} acts naturally on the Weyl symbol
modulo terms of order $\O(h^2)$ (see also~\cite{morin2019}). Here we
give a direct proof of this and compute explicitly the $\O(h^2)$
remainder.

\begin{lemma}\label{lem.A}
	Consider a change of variable $\kappa : \R^2_y\to\R^2_x$. We let	
\[U\psi=|g|^{\frac 14}\psi\circ\kappa=\textup{Jac}(\kappa)^{\frac 12}\psi\circ\kappa\,.\]
We have
	\[U\Lh U^{-1}=(-i h\nabla_y-\tilde{\mathbf{A}})g^*(-i h\nabla_y-\tilde{\mathbf{A}})-h^2V\,,\]
	with
	\[V=|g|^{-\frac 12}\left(\mathrm{div}(|g|^{\frac 14}g^*\nabla(|g|^{\frac 14}))+\|g^*\nabla(|g|^{\frac 14})\|^2\right)\,,\]
	and
	\[g^{*}=(g^{-1})^{\mathrm{T}}\,,\quad g=(\dd\kappa)^{\mathrm{T}}(\dd\kappa)\,,\quad\widetilde{\mathbf{A}}=(\dd\kappa)^{\mathrm{T}}\circ\mathbf{A}\circ\kappa\,. \]
\end{lemma}
\begin{proof}
Considering the quadratic form $\mathscr{Q}_h$ of $\Lh$ on $L^2(\R^2_x,\dd x)$, we have
\[\mathscr{Q}_h(\psi)=\int_{\R^2}\langle(-i h\nabla_y-\tilde{\textbf{A}}(y))\tilde\psi ,(-i h\nabla_y-\tilde{\textbf{A}}(y))\tilde\psi \rangle_{g^{*}}|g|^{\frac 12}\dd y\,,\]
where
\[g^{*}=(g^{-1})^{\mathrm{T}}\,,\quad g=(\dd\kappa)^{\mathrm{T}}(\dd\kappa)\,,\quad\tilde\psi=\psi\circ\kappa\,,\quad\tilde{\textbf{A}}=(\dd\kappa)^{\mathrm{T}}\circ\textbf{A}\circ\kappa\,. \]
In terms of forms, this means that
\[\kappa^*g_0=g\,,\quad \kappa^*\psi=\tilde\psi\,,\quad\kappa^*(A_1\dd x_1+A_2\dd x_2)=\tilde A_1\dd y_1+\tilde A_2\dd y_2\,.\]
We let $P=-i h\nabla_y-\tilde{\mathbf{A}}(y)$ and notice that
\[
\begin{split}
\mathscr{Q}_h(\psi)&=\int_{\R^2}\langle |g|^{\frac 14}P\tilde\psi,|g|^{\frac 14}P\tilde\psi\rangle_{g^*}\dd y\\
&=\int_{\R^2}\langle P|g|^{\frac 14}\tilde\psi,|g|^{\frac 14}P\tilde\psi\rangle_{g^*}\dd y+\int_{\R^2}\langle [|g|^{\frac 14},P]\tilde\psi,|g|^{\frac 14}P\tilde\psi\rangle_{g^*}\dd y\,,
\end{split}
\]
and then
\begin{multline*}
\mathscr{Q}_h(\psi)=\int_{\R^2}\| P|g|^{\frac 14}\tilde\psi\|^2_{g^*}\dd y\\
+\int_{\R^2}\langle [|g|^{\frac 14},P]\tilde\psi,|g|^{\frac 14}P\tilde\psi\rangle_{g^*}\dd y+\int_{\R^2}\langle P|g|^{\frac 14}\tilde\psi,[|g|^{\frac 14},P]\tilde\psi\rangle_{g^*}\dd y\,,
\end{multline*}
so that
\begin{multline*}
\mathscr{Q}_h(\psi)=\int_{\R^2}\| P|g|^{\frac 14}\tilde\psi\|^2_{g^*}\dd y\\
+2\Re\int_{\R^2}\langle [|g|^{\frac 14},P]\tilde\psi,|g|^{\frac 14}P\tilde\psi\rangle_{g^*}\dd y-\int_{\R^2}\|[|g|^{\frac 14},P]\tilde\psi\|^2_{g^*}\dd y\,.
\end{multline*}
Since $[P,|g|^{\frac 14}]=-i h\nabla(|g|^{\frac 14})$ and $\tilde{\mathbf{A}}$ is real-valued, we deduce that
\[
\begin{split}
2\Re\int_{\R^2}\langle [|g|^{\frac 14},P]\tilde\psi,|g|^{\frac 14}P\tilde\psi\rangle_{g^*}\dd y&=2h\Im\int_{\R^2}\langle \tilde\psi \nabla(|g|^{\frac 14}),|g|^{\frac 14}(-i h\nabla_y)\tilde\psi\rangle_{g^*}\dd y\\
&=2h^2\Re\int_{\R^2}\langle \tilde\psi \nabla(|g|^{\frac 14}),|g|^{\frac 14}\nabla_y\tilde\psi\rangle_{g^*}\dd y\\
&=2h^2\Re\int_{\R^2} \widetilde{\psi} (\mathbf{F}\cdot\nabla_y)\overline{\widetilde{\psi}}\dd y\\
&=h^2\int_{\R^2} \mathbf{F}\cdot(\nabla_y|\widetilde{\psi}|^2)\dd y\\
&=-h^2\int_{\R^2} \mathrm{div}\mathbf{F}\,|\widetilde{\psi}|^2\dd y\,.
\end{split}
\]
where
$\mathbf{F}=|g|^{\frac 14}g^*\nabla(|g|^{\frac 14})$. Therefore,
\[
\mathscr{Q}_h(\psi)=\int_{\R^2}\| P U\psi\|^2_{g^*}\dd y-h^2\int_{\R^2}V(y)|U\psi|^2\dd y\,,
\]
and the conclusion follows.
\end{proof}

\section*{Acknowledgments}
N.R. is deeply grateful the Mittag-Leffler Institute where this work
was started. N.R. also thanks Martin Vogel for many stimulating
discussions in the Mittag-Leffler library.

\end{document}